\pdfoutput=1
\documentclass{amsart}
\usepackage[margin=1.2in]{geometry}

\usepackage{amsmath,amssymb,mathtools}
\usepackage{dsfont,color}
\usepackage{amsthm, enumerate}
\usepackage{mathrsfs}

\providecommand\mathbb[1]{\ensuremath{\mathsf{#1}}}

\newcommand\ZZ{\mathbb{Z}}
\newcommand{\RR}{\mathbb{R}}
\newcommand{\NN}{\mathbb{N}}

\newcommand{\CC}{\mathbb{C}}
\newcommand{\FF}{\mathbb{F}}
\newcommand{\QQ}{\mathbb{Q}}

\newcommand{\PP}{\mathbb{P}}

\newcommand{\GL}{\text{GL}}
\newcommand{\M}{\text{M}}

\newcommand\rk{\text{rk}}
\newcommand\cok{\text{cok}\,}

\newcommand{\diag}{\text{diag}}

\newcommand{\T}{^{\top}}

\newcommand{\brackets}[1]{\left(#1\right)}

\renewcommand{\mod}{\text{\ mod\,}}

\newtheorem{thm}{Theorem}[section]
\newtheorem{lem}[thm]{Lemma}
\newtheorem{prop}[thm]{Proposition}
\newtheorem{cor}[thm]{Corollary} 
\newtheorem{remark}[thm]{Remark}

\title[Distribution of classes of symmetric p-adic matrices]{On the distribution of equivalence classes of random symmetric p-adic matrices}

\author{Valeriya Kovaleva}

\address{D{\'e}partment  de Math{\'e}matiques et Statistique,   Universit{\'e} de Montr{\'e}al, Montr{\'e}al, QC  H3C 3J7, Canada}
\email{valkovaleva42@gmail.com}

\begin{document}
\maketitle
\begin{abstract}

We consider random symmetric matrices with independent entries distributed according to the Haar measure on $\mathbb{Z}_p$ for odd primes $p$ and derive the distribution of their canonical form with respect to several equivalence relations. We give a few examples of applications including an alternative proof for the result of Bhargava, Cremona, Fisher, Jones, and Keating on the probability that a random quadratic form over $\ZZ_p$ has a non-trivial zero.
\end{abstract}

\section{Introduction}

Let $p$ be prime, and let $\ZZ_p$ be the ring of $p$-adic integers. Denote the set of $n \times n$ matrices over $\ZZ_p$ by $\M_n(\ZZ_p)$, and its subset of invertible matrices by $K:= \GL_n(\ZZ_p)$. For any matrix $X \in \M_n(\ZZ_p)$ there exist matrices $U, V \in K$ such that there is a decomposition 
\[
X = U \Sigma V\,,
\]
where $\Sigma = \text{diag}(p^{k_1}, \ldots, p^{k_n})$ and $k_i$ are the unique integers satisfying $0~\le~k_1~\le~\ldots~\le~ k_n\le~+~\infty$. The numbers $p^{k_j}$ define the equivalence class of $X$ with respect to left and right multiplication by $\GL_n(\ZZ_p)$. These numbers called the elementary divisors of $X$ are analogous to the singular values of complex matrices.

Let $X \in \M_n(\ZZ_p)$ be a random matrix with independent entries distributed according to the Haar measure on $\ZZ_p$. This entry-wise measure on $\M_n(\ZZ_p)$ is in fact the (multiplicative and additive) Haar measure on $\M_n(\ZZ_p)$ as a group. The elementary divisors of $X$ are then random variables, and the probability that a random matrix $X$ is equivalent to $\Sigma$ is the measure of the double coset $K \Sigma K$.

Elementary divisors come up in many different contexts, and in each one of them they exhibit a different kind of structure. First, their joint distribution was studied by Macdonald \cite{macdonald1998symmetric} from the point of view of symmetric functions and the Hecke ring of $\GL_n$. It follows from his work that for $\mathbf{k} = (k_1, \ldots, k_n)$ and $\Sigma = \text{diag} (p^{\mathbf{k}})$ the measure of the double coset $K \Sigma K$ equals
\begin{equation}\label{gensymf}
    \pi_n\ p^{-\sum_i (n-i) k_i} P_{\mathbf{k}}(p^{-1}, p^{-2}, \ldots,p^{-n};p^{-1})\,,
\end{equation}
where $\pi_n := p^{-n^2} |\GL_n(\FF_p)|$ is the density of invertible $n\times n$ matrices over a finite field $\FF_p$ and $P_{\mathbf{k}}(x_1, \ldots, x_n;t)$ is the Hall-Littlewood polynomial of the partition $\mathbf{k}$. If matrix $X$ has elementary divisors $p^{k_1},\ldots,p^{k_n}$, then its cokernel $\text{cok}\, X$ satisfies
\[
\text{cok}\, X \cong \bigoplus_{\substack{i = 1}}^n \ZZ/p^{k_i} \ZZ\,,
\]
so studying elementary divisors is equivalent to studying cokernels. Friedman and Washington \cite{friedman1987distribution} were the first to draw a connection between the distribution of cokernels of random matrices and the Cohen-Lenstra heuristics in their work on the distribution of class groups of curves over a finite field. They also showed that for a finite abelian $p$-group $H$ of $p$-rank $r$
\[
\PP\{\text{cok}\, X \cong H\} = \frac{\pi_n^2}{\pi_r\, |\text{Aut}\,(H)|}\,.
\]
This was generalised to other distributions of entries in the asymptotic case as $n$ grows by Maples \cite{maples2013cokernels} and Matchett Wood \cite{wood2019random}. 
% Evans \cite{evans2002elementary} studied the distribution of elementary divisors from the probabilistic point of view silimar to Brent and McKay \cite{brent1987determinants}, and Fulman \cite{fulman2002random}. He showed that one can construct a Markov chain from the multiplicities $m_k$ of elementary divisors, and represented the distribution in the form
% \[
% \PP\{X \in K\Sigma K\} = \frac{\pi_n^2}{\pi_{m_0} \pi_{m_1}\cdots \pi_{m_k}\cdots} p^{-(n-m_0)^2 - (n-m_0-m_1)^2 - \ldots - (n-m_0- \ldots - m_k)^2 - \ldots}\,.
% \]

Evans \cite{evans2002elementary} studied the distribution of elementary divisors from the probabilistic point of view similar to Brent and McKay \cite{brent1987determinants}, and Fulman \cite{fulman2002random}. He showed that one can construct a Markov chain from the multiplicities $m_k$ of elementary divisors, and represented the distribution
\[
\PP\{X \in K\Sigma K\} = \frac{\pi_n^2}{\pi_{m_0} \pi_{m_1}\cdots \pi_{m_k}\cdots} p^{-(n-m_0)^2 - (n-m_0-m_1)^2 - \ldots - (n-m_0- \ldots - m_k)^2 - \ldots}\,
\]
in the form
\[
\PP\{X \in K\Sigma K\} = \prod_{j=0}^\infty \frac{\pi^2_{l_{j}}}{\pi_{m_j} \pi_{l_j-m_j}^2}p^{-(l_j-m_j)^2}
\]
where $l_0 = n$ and $l_j = n-m_0 - \ldots-m_{j-1}$ for $j \ge 1$.

Finally, for a matrix $X \in K \Sigma K$ the matrix $\Sigma$ is sometimes referred as its Smith normal form, in this context see, for example, the work of Wang and Stanley \cite{wang2017smith}.

Restricting to symmetric matrices allows us to consider other similar types of decompositions. As general matrices represent bilinear forms, symmetric matrices represent quadratic forms, and this connection provides a helpful way to interpret the decompositions for symmetric matrices. 

Let $p$ be an odd prime, and let $S_n(\ZZ_p)$ be the set of symmetric $n \times n$ matrices over $\ZZ_p$. For any matrix $X \in S_n(\ZZ_p)$ there exists a matrix $U \in \GL_n(\ZZ_p)$ such that there is a decomposition 
\[
X = U \Sigma S U^\top\,,
\]
where $\Sigma = \diag(p^{k_1}, \ldots, p^{k_n})$ is the matrix of elementary divisors with $0\le k_1\le \ldots \le k_n$, and $S = \diag(a_1,\ldots,a_n)$ is the matrix of signatures. Let $r \in \ZZ_p^\times$ is a fixed non-square in $\ZZ_p$, then $a_i \in \{1,r\}$. Once put in the canonical form (depending on $\Sigma$), the signatures $a_i$ together with the elementary divisors define the $\ZZ_p$-equivalence class of $X$. This decomposition of course corresponds to diagonalising quadratic forms.

Studying $\ZZ_p$-equvalence classes is almost equivalent to studying the cokernel of $X$ with a given duality pairing. Clancy, Kaplan, Leake, Payne, and Matchett Wood \cite{clancy2015cohen} show that the probability that $X$ has a cokernel $H$ of $p$-rank $r$ with a fixed duality pairing $\delta$ equals
\begin{equation}\label{cokdual}
    \frac{\pi_n}{\beta_{n-r} |H| |\text{Aut}\,(H,\delta)|}\,,
\end{equation}
where $\beta_t := \prod_{i=1}^{\lfloor\frac{t}{2}\rfloor} (1-p^{-2i})$. The asymptotic distribution of cokernels of random symmetric $p$-adic matrices with independent entries was further studied by Matchett Wood \cite{wood2017distribution}.

Let $\QQ_p$ denote the field of $p$-adic numbers, and $\GL_n(\QQ_p)$ denote the set of all invertible matrices over $\QQ_p$. For any matrix $X \in S_n(\ZZ_p)$ there exists a matrix $U \in \GL_n(\QQ_p)$ such that there is a decomposition $X = U D U^\top$, where $D$ is the canonical form for the $\QQ_p$-equivalence class of $X$. This corresponds to the canonical form of quadratic forms over $\QQ_p$. The classification of quadratic forms over $\QQ_p$ is well known, see Cassels \cite{cassels2008rational}, but the distribution of such classes is much less studied. A related work in this direction is the paper of Bhargava, Cremona, Fisher, Jones, and Keating \cite{bhargava2016probability} on the probability of isotropy of integral quadratic forms.

In this paper we consider random symmetric matrices with entries distributed independently according to the Haar measure on $\ZZ_p$ for odd primes $p$, and compute the distribution of equivalence classes with respect to conjugation by $\GL_n(\ZZ_p)$ and $\GL_n(\QQ_p)$. This is a natural choice of measure for two reasons. First, this measure on $S_n(\ZZ_p)$ should be seen as the inverse limit of the counting measures on $S_n(\ZZ/p^k\ZZ)$ as $k \to \infty$. Second, the Haar measure on the ring of integers of a local field is analogous to the standard Gaussian measure over $\RR$ due to its invariance properties (for exposition see Evans \cite{evans2001local}). This makes such a random matrix the analogue of the Gaussian Ensemble of random matrix theory, so it is reasonable to expect that the asymptotic behaviour of such matrices is universal i.e. showcases the asymptotic behaviour of a larger class of random matrices. This turns out to be true as shown in the work of Matchett Wood \cite{wood2017distribution}. Additionally, due to the invariance of this measure on $S_n(\ZZ_p)$ with respect to conjugation by $\GL_n(\ZZ_p)$, it is an analogue of the Gaussian Orthogonal Ensemble (GOE). 
% It is also interesting to compare the distribution of the elementary divisors with the distribution of the singular values of GOE. 

Our approach to studying symmetric matrices was motivated by classical random matrix theory and the ``local GOE'' analogy rather than the Cohen-Lenstra heuristics. We believe that the idea of reducing a certain class of problems to studying the joint distribution of eigenvalues or singular values may provide a more unified framework over different fields. 

The results of this paper are as follows. First, we give an elementary random matrix theory inspired proof for the distribution of $\ZZ_p$-equivalence classes. We note again that this distribution is essentially the same as \eqref{cokdual} with a small difference of taking into account the signature for $k_j=0$, so this result though independent is not entirely new.
% \begin{thm}\label{intromain}
% Let $X$ be a random $n \times n$ symmetric matrix with independent entries distributed according to the Haar measure on $\ZZ_p$. Let $\Sigma = \diag(p^{k_1},\ldots,p^{k_n})$ be its matrix of elementary divisors and $S = \text{diag}(1,\ldots, 1,a_0, \ldots,1,\ldots,1,a_k,\ldots)$ be an admissible matrix of signatures in the canonical form. Let $m_k = \#\{i:k_i=k\}$ and $s_k$ be the quadratic characters $\big(\frac{a_k}{p}\big)$. Then
% \[
% \PP\{X \in K\Sigma S K^\top\} =
% % |\det(\Sigma)|^{\frac{n+1}{2}} p^{\frac{\Delta(D)}{2}} \frac{\Pi_n}{\rho_{\ell}^s} = 
% % \prod_{j=1}^{n} p^{-k_jn-j+1)} \frac{\Pi_n}{\rho_{\ell}^s}, 
% \frac{\pi_n}{\alpha_{m_0}^{s_0} \alpha_{m_1}^{s_1}\cdots \alpha_{m_k}^{s_k}\cdots} \prod_{j=1}^{n} p^{-k_j(n-j+1)}\,,
% \]
%  where $\alpha_{j}^s = p^{-j(j-1)/2} |\text{O}_j^s(\FF_p)|$ is the density of the orthogonal group over $\FF_p$ for $j \ge 1$, and $\alpha_{0}^+ = 1$. 
% \end{thm}
\begin{thm}\label{intromain}
Let $p$ be an odd prime, and $X$ be a random $n \times n$ symmetric matrix with independent entries distributed uniformly on $\ZZ_p$. Let $\Sigma = \diag(p^{k_1},\ldots,p^{k_n})$ be its matrix of elementary divisors and $S = \text{diag}(1,\ldots, 1,a_0, \ldots,1,\ldots,1,a_k,\ldots)$ be an admissible matrix of signatures in the canonical form. Let $m_k = \#\{i:k_i=k\}$ and $s_k$ be the quadratic character mod $p$ at $a_k$. Then
\[
\PP\{X \in K\Sigma S K^\top\} =
% |\det(\Sigma)|^{\frac{n+1}{2}} p^{\frac{\Delta(D)}{2}} \frac{\Pi_n}{\rho_{\ell}^s} = 
% \prod_{j=1}^{n} p^{-k_jn-j+1)} \frac{\Pi_n}{\rho_{\ell}^s}, 
\frac{\pi_n}{\alpha_{m_0}^{s_0} \alpha_{m_1}^{s_1}\cdots \alpha_{m_k}^{s_k}\cdots} \prod_{j=1}^{n} p^{-k_j(n-j+1)}\,,
\]
 where $\alpha_{j}^s = p^{-j(j-1)/2} |\text{O}_j^s(\FF_p)|$ is the density of the orthogonal group over $\FF_p$ for $j \ge 1$, and $\alpha_{0}^+ = 1$. 
\end{thm}
\begin{cor}\label{eldivcor}
For a random symmetric matrix defined as above
\[
\PP\{X \in K\Sigma K\} =
% |\det(\Sigma)|^{n+1/2} p^{\Delta(\Sigma)/2} \frac{\Pi_n}{A_{\ell}} = 
\frac{\pi_n}{\beta_{m_0}\cdots \beta_{m_k}\cdots} \prod_{j=1}^{n} p^{-k_j(n-j+1)}\,,
\]
where $\beta_t = ((\alpha_{j}^+)^{-1} + (\alpha_{j}^-)^{-1})^{-1} =  \prod_{i=1}^{\lfloor\frac{t}{2}\rfloor} (1-p^{-2i})$.
\end{cor}
Note that we only consider odd primes to avoid the theory of quadratic forms in characteristic $2$. A similar statement should nonetheless hold in that case as well.
% Note that $\sum_k m_k = n$, so only finitely many $m_k$ are positive, and the product in the denominator is effectively finite. Moreover, if $m_k = 0$, then $s_k = +$.

There are multiple elementary ways to prove Theorem \ref{intromain}. The one we suggest is as follows. Fix $\Sigma$ and reduce to a counting problem over $\ZZ/p^K\ZZ$ for some big $K$. Compute the size of the stabiliser of $\Sigma S$ in $\ZZ/p^K\ZZ$ by counting solutions to $U \Sigma S U^\top = \Sigma S$. This system is easy to solve over $\FF_p$, then  we can lift the solutions back to $\ZZ/p^K\ZZ$ and eventually $\ZZ_p$ via Hensel's lemma. In spirit, this is similar to computing the Jacobian of the transformation from the entry-wise to the eigenvalue-eigenvector coordinates as done for the Gaussian Orthogonal Ensemble. The proof can also be adapted to the case $p=2$, though we chose not to do that.

Further, we use the distribution of $\ZZ_p$-classes to obtain a closed expression for the distribution of $\QQ_p$-classes. This result is new. 

\begin{thm}
Let $p$ be an odd prime, and let $\varepsilon$ be the quadratic character mod $p$ at $-1$. Let $Q$ be a random quadratic form in $n$ variables with independent coefficients distributed according to the Haar measure on $\ZZ_p$. Let $d(Q) \in \QQ_p^\times/(\QQ_p^\times)^2$ be its discriminant and $c(Q) \in \{\pm 1\}$ be its Hasse invariant.  Then if $n$ is odd,
\[
\PP\{d(Q) = a,c(Q) = b\} = \begin{cases}
\frac{1}{4(1+1/p)} + b\, \frac{\pi_n}{4\beta_{n+1}\beta_{n-1}}\,, & \text{if $a \in \{1,r\}$}\,;\\
\\
\frac{1}{4p(1+1/p)} + b\, \varepsilon \big(\frac{\varepsilon}{p}\big)^{\frac{n+1}{2}} \frac{\pi_n}{4\beta_{n+1}\beta_{n-1}\,}, & \text{if $a \in \{p,pr\}$}\,;
\end{cases}
\]
and if $n$ is even and $s$ is the quadratic character mod $p$ at $a$,
\[
\PP\{d(Q) = a,c(Q) = b\} = \begin{cases}
\frac{(1+ s\brackets{\frac{\varepsilon}{p}}^{n/2}) ( 1 - \frac{s}{p^2} \brackets{\frac{\varepsilon}{p}}^{n/2})}{4(1+1/p)(1-1/p^{n+1})} + b\, \frac{\pi_n}{2\beta_{n}\alpha_{n}^s}\,, & \text{if $a \in \{1,r\}$}\,;\\
\\
\frac{1-1/p^n}{4p(1+1/p)(1-1/p^{n+1})}\,, & \text{if $a \in \{p,pr\}$}\,.
\end{cases}
\]
\end{thm}

In the last section of this paper we provide a few examples of applications for the densities obtained. First, we use the distribution of elementary divisors from Corollary \ref{eldivcor} to derive the distribution of ranks and determinants over $\ZZ/p^k\ZZ$. Second, we use the local-global principle to compute the probability that for a random $n \times n$ integer matrix, the greatest common divisor of all its minors of size $m<n$ is square-free. Equivalently, this is the probability that a random integer matrix is similar to a matrix with a square-free minor of size $m$.

\begin{cor}\label{cor:sqfminor}
Let $r,n \in \NN$ with $1\le r\le n-1$. Then as $H \to \infty$ the proportion of integer $n\times n$ symmetric matrices with entries between $1$ and $H$ such that the greatest common divisor of its minors of size $n-r$ is square-free tends to
\[
\rho_{sqf}(r,n) = \prod_p \brackets{\sum_{t=0}^{r+1} \frac{\pi_n}{\beta_{n-t}\pi_t} p^{-\frac{t(t+1)}{2}} - \frac{\pi_n}{\beta_{n-r-1}\pi_{r+1}} p^{-(r+1)(r+2)}} > 0\,.
\]
In particular, if $r = 1$ and $n \to \infty$, then
\[
\rho_{sqf}(1,n) \to \frac{\prod_p ( 1 + p^{-3}(1+1/p)^{-1} + p^{-4})}{\zeta(3)\zeta(5) \cdots \zeta(2k+1)\cdots} \approx 0.9581\,,
\]
and if $r = r(n) \to \infty$ as $n \to \infty$, then $\rho_{sqf}(r,n) \to 1$.
% In addition, the proportion of symmetric matrices with square-free determinant is
% \[
% \prod_p \brackets{\frac{\pi_n}{\beta_n} + p^{-1} \frac{\pi_n}{\beta_{n-1}}} > 0\,.
% \]
% As $n \rightarrow \infty$ this expression tends to $(\zeta(2) \zeta(3) \zeta(5) \cdots \zeta(2k+1)\cdots)^{-1} \approx 0.4824$.
\end{cor}
We note that Corollary \ref{cor:sqfminor} relies on the infinite version of the Chinese remainder theorem known as Ekedahl's sieve \cite{ekedahl1991infinite}, and thus does not include the case $r=0$. The latter corresponds to counting matrices with square-free determinant, and would require one to first prove that such a local product applies. In general, one can only count square-free values of multivariate polynomials when the number of variables is exponential in the degree of the polynomial with the best bound on the number of variables given by Destagnol and Sofos \cite{destagnol2019rational}. There are also exceptionally ``nice'' polynomials, where such a restriction does not apply, see \cite{bhargava2014geometric, bhargava2016squarefree,bhargava2022sqf}. Since the determinant is also symmetric in the entries of the matrix, and either linear or quadratic in each of them, one might hope that this problem is also more approachable, but it currently remains open.

% Poonen \cite{poonen2003squarefree} proved that such a local product applies under the ABC-conjecture, but in a weaker sense when one variable is much larger than the others, and this statement extends to multivariate polynomials with a variable of degree $\le 3$ without assuming the ABC-conjecture. Thus, In the case of the determinant of either a general or a symmetric matrix, one can notice that it is also an example of an invariant polynomial, and utilise its symmetry properties to prove the corresponding statement.

Finally, we show several ways of deriving the probability of isotropy i.e. existence of a non-trivial zero for random quadratic forms in $n\le 4$ variables. A quadratic form over $\ZZ_p$ in $n\ge 5$ is always isotropic and for the cases of $n=2,3,4$ the probability of isotropy was obtained by Bhargava, Cremona, Fisher, Jones, and Keating \cite{bhargava2016probability} through a recursive computation over finite fields. Though their method lies in a different framework of recursive computation, and is capable of solving a much bigger class of problems (see Bhargava \cite{bhargava2014positive}, Bhargava, Cremona, and Fisher \cite{bhargava2016proportion, bhargava2020proportion}), ours might be more straightforward specifically for quadratic forms.

\begin{cor}
Let $p$ be an odd prime, and let $Q$ be a random quadratic form in $n$ variables with coefficients distributed independently according to the Haar measure on $\ZZ_p$. Then
\[
\PP\{\text{$Q$ is isotropic}\} = \begin{cases}
\frac{1}{2}\,, & \text{if $n=2$}\,;\\
1 - \frac{1}{2p(1+ 1/p)^2}\,, & \text{if $n=3$}\,;\\
1 - \frac{1-1/p}{4p^3(1+1/p)^2(1-1/p^5)}\,, & \text{if $n=4$}\,.\\
\end{cases}
\]
\end{cor}

We would like to add a bit more historical context on the development of random matrix theory over different fields. Random matrix theory over fields other than complexes has mostly been evolving independently of the theory over the complex numbers. For an overview of random matrix theory over finite fields and in particular cycle index techniques see Fulman \cite{fulman2002random}, and for an overview of rank problems and Stein's method see Fulman and Goldstein \cite{fulman2015stein}. One of the recent results exploring the similarities between theory over different fields is the work of Gorodetsky and Rodgers \cite{gorodetsky2019traces}, on the finite field analogue of the fundamental result of Diaconis and Shahshahani \cite{diaconis1994eigenvalues} on the convergence of traces of unitary matrices to the normal distribution. Ideas used in \cite{diaconis1994eigenvalues} rely heavily on representation theory and symmetric function theory. In our case, we may also observe that random $p$-adic matrices are intimately connected with Hall-Littlewood polynomials, which are a generalisation of Schur polynomials used in the work of Diaconis and Shahshahani. Optimistically, we may even expect that with appropriate normalisation results over $p$-adics should morph into the Euclidean case as $p \to \infty$, see for example \cite{van2021limits}.

Random $p$-adic matrices have been studied to a smaller extent. Apart from the ones already mentioned, some of the recent developments in this area are the works of Ellenberg, Jain, and Venkatesh \cite{ellenberg2011modeling} on modelling $\lambda$-invariants, and Poonen and Rains \cite{poonen2012random}, Bhargava, Kane, Lenstra, Poonen, and Rains \cite{bhargava2013modeling}, Poonen \cite{poonen2018heuristics}, Park, Poonen, and Matchett Wood \cite{park2016heuristic} on the heuristics for ranks of elliptic curves based on random symmetric and alternating matrices. In light of these efforts it makes sense to try and better understand the behaviour of different random matrix models on the structural level. More precisely, one may ask what structural properties random matrices over $\FF_p$ or $\ZZ_p$ share with the objects they are modelling. 

We expect that as long as a matrix group has a suitable decomposition, ideas presented in this paper would apply. An obvious example here is the additive group of skew-symmetric (alternating) matrices
\[
A_n(R) = \{X \in \M_n(R): X = -X^\top\}\,,
\]
or similarly, the additive group of symmetric matrices with zero diagonal. In these cases, our approach to computing the distribution of $\ZZ_p$ classes should adapt straightforwardly.

\section{Preliminaries}

Let us first discuss general $p$-adic matrices without the symmetry constraint. Let $\M_n(\QQ_p)$ be the set of all $n \times n$ matrices with entries in $p$-adic numbers $\QQ_p$, and $\M_n(\ZZ_p)$ be its subset of matrices with integral entries. Let $\nu$ be the Haar measure on $\QQ_p$ normalised so that $\nu(\ZZ_p) = 1$, and let $\mu$ be the measure on $\M_n(\QQ_p)$ such that it is the direct product of $n^2$ measures $\nu$ on its entries. We also trivially have that $\mu \brackets{\M_n(\ZZ_p)} = 1$.

The general linear group $\GL_n(\QQ_p)$ is the set of invertible $n \times n$ matrices with entries in $\QQ_p$. The general linear group $K := \GL_n(\ZZ_p)$ is its subgroup of invertible $n \times n$ matrices over $\ZZ_p$, or
\[
\GL_n(\ZZ_p) := \{A \in \text{M}_n(\ZZ_p):\ \det(A) \in \ZZ_p^{\times}\} = \{A \in \text{M}_n(\ZZ_p):\ |\det(A)|_p = 1\}\,.
\]
As a set $\GL_n(\ZZ_p)$ is both open and compact of measure
\[
\pi_n :=\mu\big(\GL_n(\ZZ_p)\big) =
% |\GL_n(\FF_p)| p^{-n^2} =
\prod_{k=1}^n\big(1-p^{-k}\big)\,.\]
As a group $\GL_n(\ZZ_p)$ is the largest compact subgroup of $\GL_n(\QQ_p)$, and in this sense it corresponds to the unitary group over Euclidean spaces. Unitary matrices represent isometries, and preserve the Frobenius norm of a complex matrix, while the matrices from $\GL_n(\ZZ_p)$ instead preserve the max-norm $\|A\|_p = \max_{i,j} |a_{ij}|_p$.

\begin{lem}
Let $A \in \M_n(\QQ_p)$ and $\gamma \in \GL_n(\ZZ_p)$. Then $\|\gamma A\|_p = \|A\gamma\|_p = \|A\|_p$.
\end{lem}

\begin{proof}
For any $\gamma \in \GL_n(\ZZ_p)$ we have $\|\gamma\|_p = \|\gamma^{-1}\|_p = 1$. Let $B = \gamma A$, then for any $b_{ij} = \sum_{k=1}^n \gamma_{ik}a_{kj}$ we have
\[
|b_{ij}|_p \le \max_k |\gamma_{ik}a_{kj}|_p \le \max |\gamma_{ik}|_p \max |a_{kj}|_p \le \|\gamma\|_p \|A\|_p = ||A||_p\,,
\]
hence $\|B\|_p \le \|A\|_p$. Conversely, for $A = \gamma^{-1} B$, we also have $\|A\|_p \le \|\gamma^{-1}\|_p \|B\|_p =  \|B\|_p$, so $\|B\|_p = \|A\|_p$.
\end{proof}

Any complex matrix has a canonical form obtained via the singular value decomposition. This corresponds to factoring out by the unitary group $\text{U}_n(\CC)$, which is a maximal compact subgroup of $\GL_n(\CC)$. For any matrix $A \in \M_n(\CC)$ there exist matrices $U,V \in \text{U}_n(\CC)$ such that $A = U S V^*$, where $S$ is the diagonal matrix with entries $\sigma_1 \ge \sigma_2 \ge \ldots \ge \sigma_n \ge 0$ defined uniquely. These are called the singular values of $A$. Similarly, for any $A \in \M_n(\QQ_p)$ there exists a decomposition $A = U \Sigma V$,
where $U, V \in \GL_n(\ZZ_p)$, and $\Sigma = \text{diag}(p^{k_1}, \ldots, p^{k_n})$ with $k_i$ integers such that $-\infty < k_1 \le  \ldots \le k_n \le +\infty$. This sequence is unique and is called the sequence of elementary divisors. Note that $k_i = +\infty$  is possible with $p^{+\infty} = 0$ since $|p^{+\infty}|_p = p^{-\infty} = 0$. In this case the matrix is singular. Such a decomposition exists more generally over principal ideal rings, and is sometimes called the Smith normal form. For general theory see Brown's book \cite{brown1993matrices}.

If $X \in \M_n(\ZZ_p)$, then all $k_i$ are non-negative. Moreover, we can also assume that $k_i < +\infty$, since the set of singular matrices $V = \{\mathbf{X} \in \text{M}_n(\QQ_p): \det \mathbf{X} = 0\}$ has measure zero. 
% The proof of this fact is analogous to the case of $\RR$ and Lebesgue's measure and relies on the fact that the graph of a rational function defined over an open set has measure zero. 
So further we will restrict to the set of non-singular matrices $\tilde{\M}_n(\ZZ_p) := \M_n(\ZZ_p) \cap \GL_n(\QQ_p)$. Now let us give a simple proof for the distribution of elementary divisors for general matrices over $\ZZ_p$.

% \begin{lem}\label{gamlem}
% For any measurable subset $\mathcal{A} \subseteq \M_n(\QQ_p)$ with finite measure and $\gamma \in \M_n(\QQ_p)$ we have $\mu(\gamma \mathcal{A}) = \mu(\mathcal{A}) |\det(\gamma)|^n.$
% \end{lem}
% \begin{proof}
% For isometries $\gamma \in \GL_n(\ZZ_p)$ we have $\mu(\gamma \mathcal{A}) = \mu(\mathcal{A})$. Due to the elementary divisor decomposition it suffices to consider $\gamma = \diag(p^{k_1},\ldots,p^{k_n})$. Left-multiplying by a diagonal matrix corresponds to row-wise multiplication by its diagonal entries. Since $\mu$ is a direct product of the Haar measures, 
% \[
% \mu\big(\gamma \mathcal{A}\big) = \mu(\mathcal{A}) \prod_k (p^{-k})^{m_k n} = \mu(\mathcal{A}) |\det(\gamma)|^n\,.
% \] 
% The space $\M_n(\QQ_p)$ is an $n^2$-dimensional space over $\QQ_p$. For a measurable set $\mathcal{A}$ and a linear operator $T:\M_n(\QQ_p) \rightarrow \M_n(\QQ_p)$ corresponding to multiplication by $\gamma$ we have
% \[
% \mu(T\mathcal{A}) = |\det T|\  \mu(\mathcal{A}) = |\det \gamma|^c\, \mu(\mathcal{A})\,.
% \]
% If $\gamma$ is diagonal, $T$ is an $n^2 \times n^2$ diagonal matrix with entries of $\gamma$ repeated $n$ times, so $c=n$. \end{proof}

For $\gamma \in \tilde{\M}_n(\ZZ_p)$ define $C_{\gamma} := K \cap \gamma K \gamma^{-1}$. In particular, for $\Sigma = \diag(p^{k_1},\ldots,p^{k_n})$ with multiplicities $m_k = \#\{i:k_i=k\}$ we have
\[
C_{\Sigma} = \begin{bmatrix}
\GL_{m_0}(\ZZ_p) & \M_{m_0\times m_1}(\ZZ_p)  & \M_{m_0\times m_2}(\ZZ_p)  &&  &\cdots \\
\M_{m_1\times m_0}(p\ZZ_p) & \GL_{m_1}(\ZZ_p) & \M_{m_1\times m_2}(\ZZ_p)  & &  &\cdots \\
\M_{m_2\times m_0}(p^2 \ZZ_p)& \M_{m_2\times m_1}(p\ZZ_p) & \GL_{m_2}(\ZZ_p) & &  &\cdots  \\
\vdots & \vdots &  & \ddots & & \\
\M_{m_k\times m_0}(p^k\ZZ_p) & \M_{m_k\times m_1}(p^{k-1}\ZZ_p) & \cdots &\cdots  &&  \GL_{m_k}(\ZZ_p) & \\
\vdots & \vdots & & & &\vdots & \ddots\\
\end{bmatrix}\,.
\]
This set is also open and compact, and has measure
\begin{equation}\label{stabsize}
   \mu\big(C_{\Sigma}\big) = p^{-D(\Sigma)} \prod_{k=0}^{\infty} \pi_{m_k}\,, 
\end{equation}
where $D(\Sigma) = \sum_{0 \le i < j} m_i m_j(j-i) = \sum_{1\le i < j \le n} (k_j-k_i)$.

% \begin{lem}\label{stablem}
% Let $\Sigma = \diag(p^{k_1},\ldots, p^{k_n})$ be a matrix of elementary divisots with multiplicities $m_k = \#\{i : k_i = k\}$, then
% \[
% \mu\big(C_{\Sigma}\big) = p^{-D(\Sigma)} \prod_{k=0}^{\infty} \pi_{m_k}\,,
% \]
% where $D(\Sigma) = \sum_{0 \le i < j} m_i m_j(j-i) = \sum_{1\le i < j \le n} (k_j-k_i)$.
% \end{lem}
% \begin{proof}
% Divide a matrix $A \in C_{\Sigma}$ into blocks $A_{i,j}$ of size $m_i \times m_j$. Using that the entries of $A$ are independent, obtain
% \[
% \mu(C_{\Sigma}) = \prod_{0 \le i < j} p^{-m_i m_j(j-i)} \prod_{k=0}^{\infty} \pi_{m_k} = \prod_{1 \le i < j \le n} p^{-(k_j-k_i)} \prod_{k=0}^{\infty} \pi_{m_k} \,,
% \]
% as $A_{i,i} \in \GL_{m_i}(\ZZ_p)$, $A_{i,j} \in \M_{m_i \times m_j}(\ZZ_p)$ for $i>j$, and $A_{i,j} \in p^{j-i}\M_{m_i \times m_j}(\ZZ_p)$ for $i<j$.
% \end{proof}

\begin{prop}\label{gendensity}
Let $X$ be a random $n \times n$ matrix with independent entries distributed according to the Haar measure on $\ZZ_p$. Let $\Sigma = \diag\{p^{k_1},\ldots, p^{k_n}\}$, and $m_k = \#\{k_i:k_i=k\}$. Then
\[
\PP\{X \in K\Sigma K\} =  \frac{\pi_n^2}{\pi_{m_0}\cdots \pi_{m_k}\cdots} |\det \Sigma|_p^n p^{D(\Sigma)} \, .
\]
\end{prop}
\begin{proof} 
The group $K = \GL_n(\ZZ_p)$ is a subgroup of the group $\tilde{\M}_n(\ZZ_p)$, so $\tilde{\M}_n(\ZZ_p)$ is a disjoint union of double $K$-cosets. For each such coset we have
\[
K \Sigma K = \bigcup_{\gamma \in K} \gamma \Sigma K = \bigsqcup_{\gamma \Sigma K \in K\Sigma K/K} \gamma \Sigma K\,.
\]
The number of left cosets of $K$ equals $[K : C_{\Sigma}]$ with $C_\Sigma = K \cap \Sigma K \Sigma^{-1}$, not necessarily finite. Let $F_{\Sigma}$ be a fundamental domain for right multiplication by $C_{\Sigma}$, then
\[
K = \bigsqcup_{\gamma \in F_{\Sigma}} \gamma C_{\Sigma}\,.
\]
Sets $\gamma C_{\Sigma}$ are open and compact, so $\bigsqcup_{\gamma \in F_{\Sigma}} \gamma C_{\Sigma}$ is a disjoint open-compact cover of the open-compact set $K$. Then this cover has a finite subcover, and it coincides with the cover itself. Using that $F_{\Sigma}$ is finite and $\gamma \in \GL_n(\ZZ_p)$ are isometries we get
\[
\mu(K) = \sum_{\gamma \in F_{\Sigma}} \mu(\gamma C_{\Sigma}) = \sum_{\gamma \in F_{\Sigma}} \mu( C_{\Sigma}) = |F_{\Sigma}|\ \mu( C_{\Sigma})\,.\]
Thus $|F_{\Sigma}| = [K : C_{\Sigma}] = [K : K \cap \Sigma K \Sigma^{-1}] = \mu(K)/\mu(C_{\Sigma})$ and 
\[
\mu(K \Sigma K) = \sum_{\gamma \Sigma K \in K\Sigma K/K} \mu(\gamma \Sigma K) = \sum_{\gamma \Sigma K \in K\Sigma K/K} \mu(\Sigma K) = |F_{\Sigma}|\ \mu(\Sigma K) = \frac{\mu(K)\mu(\Sigma K)}{\mu(C_{\Sigma})}\,.
\]
Finally, for any measurable subset $\mathcal{A} \subseteq \M_n(\QQ_p)$ with finite measure and $\gamma \in \M_n(\QQ_p)$ we have $\mu(\gamma \mathcal{A}) = \mu(\mathcal{A}) |\det(\gamma)|^n.$ This fact together with \eqref{stabsize} implies the statement of the proposition.
\end{proof}

The proof easily extends to rectangular matrices. Let $n > m$ and let $K_{n,m} \subseteq \M_{n\times m}(\ZZ_p)$ denote the set of $n\times m$ matrices consisting of $m$ linearly independent mod $p$ columns. Then for $X \in \M_{n\times m}(\ZZ_p)$ there exist $U \in K_{n,m}$ and $V \in K_m = \GL_m(\ZZ_p)$ such that $X = U \Sigma V$ with $\Sigma = \diag (p^{\mathbf{k}})$. In this case 
\[
\mu(K_{n,m} \Sigma K_m) = \frac{\mu(K_{n,m}) \mu(\Sigma K_m)}{\mu(C_{\Sigma})} =  \frac{\pi_n \pi_m}{\pi_{n-m} \pi_{m_0} \pi_{m_1} \cdots\pi_{m_k}\cdots}p^{D(\Sigma)} |\det \Sigma|_p^m\,.
\]
Moreover, the argument works for any $p$-adic number field with $\ZZ_p$ replaced by the corresponding ring of integers and $\FF_p$ replaced by the corresponding residue field. This has a nice immediate consequence.

\begin{cor} \label{genrank}Let $n, m \in \NN$ with $ n\le m$. Let $X$ be a random $n \times m$ matrix with independent entries distributed uniformly on $\FF_q$, then
\[
\PP\{\rk\, X = n - r\} = q^{-r(m-n+r)} \frac{\pi_n(q)\pi_m(q)}{\pi_r(q) \pi_{m-n+r}(q) \pi_{n-r}(q)}\,,
\]
where $\pi_j(q) = \prod_{i=1}^j (1 - q^{-i})$. In particular, when $n = m$,
\[
\PP\{\rk\, X = n - r\} = q^{-r^2} \frac{\pi_n(q)^2}{\pi_r(q)^2 \pi_{n-r}(q)}\,.
\]
\end{cor}

\begin{remark}\normalfont
The proof of Proposition \ref{gendensity} relies on the invariance of $\mu$ under multiplication by $\gamma \in\GL_n(\ZZ_p)$. The Haar measure on $p^N \ZZ_p$ is the only entry-wise measure with this property, just like the Gaussian measure over $\RR$ \cite{evans2001local} in the case of the Gaussian Orthogonal Ensemble. However it is possible to define a measure on the whole group $\M_n(\ZZ_p)$ with this property e.g. the (multiplicative) Haar measure $\mu_H$ with $d_H X = \prod \mu(d x_{ij})/|\det X|^n$, or simply define a measure that only depends on $\Sigma$, see Fulman and Kaplan \cite{fulman2019random}.
\end{remark}

Macdonald \cite{macdonald1998symmetric} showed that for the (multiplicative) Haar measure $\mu_H$ on $\M_n(\ZZ_p)$ normalised such that $\mu_H(K) = 1$
\[
\mu_H(K\Sigma K) = p^{-\sum_i (n-i) k_i} P_{\mathbf{k}}(p^{n-1}, p^{n-2}, \ldots,1;p^{-1})\,,
\]
where $P_{\lambda}(x_1,\ldots,x_n;t)$ is the Hall-Littlewood polynomial defined by
\[
P_{\lambda}(x_1,\ldots,x_n;t) = \frac{(1-t)^n}{\prod_{i=0}^{\infty} \prod_{j=1}^{m_i(\lambda)} (1-t^j)} \sum_{\sigma \in S_n} \sigma\brackets{\mathbf{x}^{\lambda} \prod_ {i <j} \frac{x_i - tx_j}{x_i-x_j}}\,.
\]
We have $\mu(K\Sigma K) = \pi_n |\det \Sigma|^n \mu_H(K\Sigma K)$, and it is also easy to see from the definition above that $A^{-|\lambda|} P_{\lambda}(x_1,\ldots,x_n;t) = P_{\lambda}(x_1/A,\ldots,x_n/A;t)$. Then 
\[
\mu(K\Sigma K) = 
% \pi_n |\det \Sigma|^n p^{-N(\mathbf{k})} P_{\mathbf{k}}(p^{n-1}, p^{n-2}, \ldots,1;p^{-1}) = 
\pi_n \,p^{-\sum_i (n-i) k_i}  P_{\mathbf{k}}(p^{-1}, p^{-2}, \ldots,p^{-n};p^{-1})\,.
\]

We find it helpful to think of the joint distribution of elementary divisors both in terms of the measure of the quotient, and in terms of symmetric functions. Both of these approaches are mirrored in random matrix theory. The first one corresponds to computing the Jacobian of the change of variables from the entries of the matrix to the parameters of the decomposition.
% ; here instead of the Jacobian we have the Radon-Nykodim derivative. 
The second one is essentially the Representation Theory approach, which is more recent, but has lead to many deep insights. For reference, see work of Diaconis and Shahshahani \cite{diaconis1994eigenvalues}, and the exposition by Gamburd \cite{gamburd2007some}.

From now on, we will restrict to odd primes $p$.

A quadratic form over $\ZZ_p$ in $n$ variables is a homogeneous multivariate polynomial of degree $2$
\[
Q(x_1,\ldots,x_n) = \sum_{1 \le i \le j \le n} q_{ij}x_i x_j, \quad q_{ij} \in \ZZ_p\,.
\]
Let $S_n(\ZZ_p)$ be the set of symmetric $n \times n$ matrices over $\ZZ_p$. This is an $n(n+1)/2$-dimensional module over $\ZZ_p$ isomorphic to the module of quadratic forms over $\ZZ_p$ in $n$ variables. We will use symmetric matrices and quadratic forms interchangeably. Due to homogeneity it does not matter if we consider quadratic forms over $\ZZ_p$ or $\QQ_p$, but the former is more convenient.

We will use the following notation throughout the paper. Let $r \in \ZZ_p^{\times}$ be a fixed non-square in $\ZZ_p$; it is also a non-square mod $p$. Define
\[
\chi(x) := \begin{cases}+1\,, & \text{if $x \in \ZZ_p^{\times} \cap (\ZZ_p)^2$}\,;\\
-1\,,& \text{if $x \in \ZZ_p^{\times} \setminus (\ZZ_p)^2$}\,;\\
0\,,& \text{if $x \in p\ZZ_p$}\,.
\end{cases}
\]
In particular, $\chi(1) = +1$ and $\chi(r) = -1$. Sometimes we will write $\chi(1) = +$ and $\chi(r) = -$ to match the notation of $\alpha_n^{\pm}$. Denote $\chi(-1)$ by $\varepsilon$. 

Let $\left<a, b\right>$ denote the Hilbert symbol
\[
\left<a, b\right> = \begin{cases}1\,, & \text{$ax^2 + by^2 = z^2$ has a non-trivial solution over $\QQ_p$}\,;\\
-1\,,& \text{otherwise}\,.
\end{cases}
\]

A non-singular $p$-adic quadratic form $Q$ has two invariants: the discriminant $d(Q) \in \QQ_p^{\times}/(\QQ_p^{\times})^2 \cong \{1,r,p,pr\}$, and the Hasse invariant $c(Q) \in \{\pm 1\}$. If $Q(\mathbf{x}) = a_1 x_1^2 + \ldots +a_n x_n^2$, then
\[
d(Q) = \prod_{i=1}^n a_i\,,\quad c(Q) = \prod_{1\le i<j \le n} \left<a_i, a_j\right>\,.
\]
These two invariants completely define the equivalence class of $Q$ with respect to a non-degenerate change of variables. Translating into the matrix language, for any matrix $Q \in S_n(\QQ_p)$ there exists a matrix $U \in \GL_n(\QQ_p)$ such that $Q = U D U^{\top}$, where $D$ is the representative of the equivalence class defined by $d(Q)$ and $c(Q)$. There are $2|\QQ_p^{\times}/(\QQ_p^{\times})^2| = 8$ classes altogether. 

Note that this is very different from the spectral decomposition of a real symmetric matrix, where the diagonal entries of $D$ are the eigenvalues of the matrix, and $U$ is the unitary matrix of eigenvectors. It is possible to consider a different decomposition which would be analogous to the spectral decomposition (for example, see \cite{fulman2002random}), but this is not the subject of this paper.

If we restrict to $U \in \GL_n(\ZZ_p)$, we have to account for scaling by powers of $p$. For any matrix $Q \in S_n(\ZZ_p)$ there exists $U \in \GL_n(\ZZ_p)$ such that $Q = U \Sigma S U^\top$, where $\Sigma$ is the matrix of elementary divisors, and $S=S(\Sigma)$ is the matrix of signatures. Each diagonal block in $\Sigma$ corresponding to $p^k$ with multiplicity $m_k$ has its own quadratic class $1$ or $r$, which is reflected in the corresponding diagonal block of $S$. 

The canonical form for a block of $S$ with  signature $s_k = +$ meaning quadratic class $1$ is the $m_k \times m_k$ identity matrix $\mathbf{1}_{m_k}^+$. For a block with signature $s_k=-$ meaning quadratic class $r$ it is the matrix $\mathbf{1}_{m_k}^-$, the $m_k \times m_k$ identity matrix with last entry replaced by $r$. If $m_k=0$, set $s_k = +$. This is exactly the same as the canonical form for quadratic forms, for more detail see Cassels' book \cite{cassels2008rational}.  Further we will write $K \Sigma S K\T$ for the $\ZZ_p$-equivalence class of $\Sigma S$ in $S_n(\ZZ_p)$. 

\section{Distribution of equivalence classes of symmetric matrices}
\subsection{$\ZZ_p$-equivalence}

In this section we consider symmetric matrices with independent entries distributed according to the Haar measure on $\ZZ_p$ for odd $p$; denote this measure on $S_n(\ZZ_p)$ by $\mu$. We derive the distribution of equivalence classes with respect to conjugation by $\GL_n(\ZZ_p)$. It is possible to consider $S_n(\ZZ_2)$ and obtain a similar answer with the principal difference coming from the canonical form and the orthogonal group over $\FF_2$.

% We again take the entries of a symmetric matrix to be independent random variables distributed according to the Haar measure on $\ZZ_p$, and call $\mu$ the corresponding measure on $S_n(\ZZ_p)$. 
% It is possible to adapt the argument from Proposition \ref{gendensity} of Section $2$, but we choose to phrase the argument differently. Namely, we are going to compute the measure of the class $K \Sigma S K^{\top}$ mod $p^k$ for some big $k$. 

First, let us show that we can reduce the problem over $\ZZ_p$ to computing $K \Sigma S K^{\top}$ mod $p^k$ for some big $k$. For convenience let $R_k := \ZZ/p^k \ZZ$. If two matrices are equivalent over $\ZZ_p$, their reductions are equivalent over $R_k$ for all $k \in \NN$. The converse is only true for $k$ large enough.
 
\begin{lem}[Chapter 8, Lemma 5.1 \cite{cassels2008rational}] \label{lifteq}
Let $A,B \in S_n(\ZZ_p)$ such that $|\det A|_p = |\det B|_p = p^{-m}$. Suppose that for $k \ge m+1$ we have $A \sim_{\GL_n(R_k)} B \mod p^k$. Then $A \sim_{\GL_n(\ZZ_p)} B$.
\end{lem}

% To compute the size of the equivalence class of a fixed $\Sigma$ mod $p^k$, use the Orbit-Stabiliser theorem. 
Let $C_R(A): = \{\gamma \in \GL_n(R): \gamma A\gamma\T = A\} \subseteq \GL_n(R)$ denote the stabiliser of $A \in \M_n(R)$, and let $\text{O}_{n}^+(R) := C_R(\mathbf{1}_n^+)$ denote the orthogonal group over a ring $R$. We only consider $R = \ZZ_p$ or $R_k = \ZZ/p^k \ZZ$, so $R$ always contains a non-square, and we can define the twin group of the orthogonal group $\text{O}_n^-(R) = C_R(\mathbf{1}_n^-)$. We will call both of them orthogonal groups. If $\FF_p$ is a finite field with $p$ odd, $\varepsilon = \chi(-1)$ and $s \in \{\pm\}$, then
\[
\#\text{O}_{n}^s(\FF_p) =  |\text{O}_{n}^s(\FF_p)| = \begin{cases}
2 p^{n(n-1)/2} \prod_{i=1}^{t}(1 - p^{-2i})\,, & \text{for $n = 2t+ 1$}\,;\\
\\
\frac{2p^{n(n-1)/2}}{1 + s \varepsilon^tp^{-t}} \prod_{i=1}^{t}(1 - p^{-2i})\,, & \text{for $n = 2t$}\,.
\end{cases}
\]
The size of the orthogonal groups over $\ZZ/p^k\ZZ$ can be obtained via multivariate Hensel's lemma, and in particular, their ``densities'' $\alpha_{n}^s := |\text{O}_{n}^s(\FF_p)|/p^{n(n-1)/2}$ are preserved.

\begin{lem}[Multivariate Hensel's lemma for systems \cite{fisher1997note}]
Let $\mathbf{f} = (f_1, \ldots, f_m) \in \ZZ_p[x_1, \ldots, x_n]^m$ with $n \ge m$, and let $\mathbf{J}_\mathbf{f} = (\partial f_i/\partial x_j)_{1 \le i \le m, 1 \le j \le n}$ be the Jacobian of the system $\mathbf{f} = \mathbf{0}$. Let $\mathbf{t} \in \ZZ_p^n$ satisfy
\[
\|\mathbf{f}(\mathbf{t})\|_p < 1\,,
\]
and suppose that there exists a subset of rows $S  \subseteq [n]$ of size $m$ such that $\mathbf{J}_{\mathbf{f},S} = (\partial f_i/\partial x_j)_{1 \le i \le m, j \in S}$, the restriction of $\mathbf{J}_\mathbf{f}$ to the set of rows in $S$, satisfies
\[
|\det \mathbf{J}_{\mathbf{f},S}(\mathbf{t})|_p = 1\,. 
\]
Then there exists $\mathbf{y} \in \ZZ_p^n$ such that $\mathbf{f}(\mathbf{y}) = 0$ with $\|\mathbf{y} - \mathbf{t}\|_p < 1$. Moreover, with variables outside of $S$ fixed, such $\mathbf{y}$ is unique.
\end{lem}
\begin{lem} \label{denslem} Let $n,k \in \NN$ and $s \in \{\pm\}$. Then  $p^{-kn(n-1)/2}|\text{O}_{n}^s(R_k)| = \alpha_n^s$.
\end{lem}

\begin{proof}
Without loss of generality, set $s = +$ and write $\mathbf{1}$ for $\mathbf{1}_n^+$ to simplify notation in this lemma. If $U \in \text{O}_n^+(R_k)$, then $U$ is a solution to the equation $X X^\top = \mathbf{1}$.
% , so we need to prove that each solution of $X X\T - \mathbf{1} = 0$ mod $p$ has exactly $p^{(k-1)n(n-1)/2}$ lifts to $R_k$. 
Let $X = (x_{ij})$, then $X X\T = \mathbf{1}$ is a system of $m = n(n+1)/2$ equations in $n^2$ variables. Enumerate the variables along the rows so that 
\[
y_1 = x_{11},\, \ldots\,, y_n = x_{1n}, x_{n+1} = x_{21}, \ldots, y_{2n} = x_{2n},\, \ldots\,, y_{n^2} = x_{nn}\,.
\]
Enumerate the equations along the rows starting from the diagonal so that 
\[
f_1 = (X X\T - \mathbf{1})_{11}, \,\ldots\,, f_n = (X X\T - \mathbf{1})_{1n}, 
% f_{n+1} = (X X\T - \mathbf{1})_{22},\, \ldots\,, 
f_{n(n+1)/2} = (X X\T - \mathbf{1})_{nn}\,.
\]
Let $A$ be a solution to $X X\T = \mathbf{1}$ mod $p$. The condition $\max_S |\det \mathbf{J}_{\mathbf{f},S}(\mathbf{t})|_p = 1$ is equivalent to $\text{rk$_p$}\, \mathbf{J}_{\mathbf{f}}(\mathbf{t}) := (\rk\, \mathbf{J}_{\mathbf{f}}(\mathbf{t}) \mod p) = m$, or that all $m$ rows of the Jacobian matrix at $A$ are linearly independent mod $p$. Let $a_1, \ldots, a_n$ be the rows of $A$. Then $\mathbf{J}_{\mathbf{f}}(A)$ is block-upper-triangular.
% the Jacobian matrix of size $n^2 \times n(n+1)/2$ is
% \[
% \mathbf{J}_{\mathbf{f}}(A) =
% \begin{bmatrix}
% 2a_1 &   &   &   &   &   \\
% a_2 &   &   &   &   &   \\
% \vdots &   &   &   &  &   \\
% a_n &    &   &   &   &  \\
% 0 & 2a_2  &   &   &   &   \\
% 0 & a_3 &   &   &   &  \\
% \vdots & \vdots &   &   &   &  \\
% 0 & a_n &   &   &   &  \\
% % 0 & 0 &  2a_3 &   &   &  \\
% \vdots & \vdots & &  &   &  \\
% 0 & 0 & & \cdots &  & 2a_n\\
% \end{bmatrix}\,.
% \]
Let $A \setminus a_1 \ldots a_j$ be the truncated $A$ without first $j$ rows. As $p \neq 2$, $\text{rk$_p$}\, \mathbf{J}_{\mathbf{f}}(\mathbf{t}) = \text{rk$_p$}\, B$ for
\[
B = \begin{bmatrix}
A & \\
0 & A \setminus a_1 &  \\
0 & 0 & A \setminus a_1, a_2 & \\
\vdots & \vdots & \vdots & \ddots &   &  \\
0 & 0 & 0 & \cdots & 0 & a_n\\
\end{bmatrix}\,.
\]
Let $b_j$ denote the rows of $B$. Suppose $\sum_j \gamma_j b_j = 0$ for some $\gamma_j \in \FF_p$. Looking only at the first $n$ columns this implies $\sum_{j\le n} \gamma_j a_j = 0$, but the rows of $A$ are linearly independent mod $p$, hence $\gamma_j = 0$ for $j \le n$. Repeating the argument for the next $n$ columns and so on conclude that $\gamma_j = 0$ for all $j$. Hence rows of $B$ are linearly independent and $\text{rk$_p$\,} \mathbf{J}_{\mathbf{f}}(\mathbf{t}) = \text{rk$_p$\,} B = m$.

For $A \in \text{O}_n^s(\FF_p)$ there is a subset of columns $S$ of $\mathbf{J}_{\mathbf{f}}(A)$ such that these columns are independent mod $p$. This corresponds to choosing $|S| = m$ variables. Lift all the variables outside of $S$ arbitrarily with $p^{(k-1)n(n-1)/2}$ options. The $m$ variables we chose have a unique lift by Hensel's lemma.
\end{proof}
\begin{thm}\label{symdensity}
Let $p$ be an odd prime, and let $X \in S_n(\ZZ_p)$ be a random symmetric $n \times n$ matrix with independent entries distributed according to the Haar measure on $\ZZ_p$. Let $\Sigma = \diag(p^{k_1},\ldots,p^{k_n})$ be a matrix of elementary divisors with multiplicities $m_k = \#\{i:k_i=k\}$, and $S = \text{diag}(\mathbf{1}_{m_0}^{s_0},\ldots,\mathbf{1}_{m_k}^{s_k},\ldots)$ be a matrix of signatures, then the probability that $X$ is in the $(\Sigma, S)$-class equals
% Let $\Sigma = \diag(p^{k_1},\ldots,p^{k_n})$ be a matrix of elementary divisors with multiplicities $m_k = \#\{i:k_i=k\}$, and $S = \text{diag}(\mathbf{1}_{m_0}^{s_0},\ldots,\mathbf{1}_{m_k}^{s_k},\ldots)$ be a matrix of signatures, then
\[
\mu(K\Sigma S K^{\top}) = \frac{\pi_n}{\alpha_{m_0}^{s_0} \cdots \alpha_{m_l}^{s_l} \cdots} \prod_{j=1}^{n} p^{-k_j(n-j+1)} \,. 
\]
\end{thm}
\begin{proof}
Let $\Sigma = \text{diag }(p^{k_1}, \ldots, p^{k_n})$ with multiplicities $m_k = \#\{i: k_i = k\}$. Let $S$ be the matrix of signatures with $s_0, s_1, \ldots$ corresponding to blocks of size $m_0, m_1$ and so on. Let $|\det \Sigma|_p = p^{-M}$ and $k_n = t$. Then by Lemma \ref{lifteq} for $m \ge M + 1$
\[
\mu(K\Sigma S K\T) = \frac{\#\{\gamma \in S_n(R_m):\ \gamma \in K \Sigma SK\T \mod p^m\}}{p^{mn(n+1)/2}}\,.
\]
By the Orbit-Stabiliser Theorem,
\[
\#\{\gamma \in \M_n(R_m):\ \gamma \in K\Sigma SK\T \mod p^m\} = \frac{|\GL_n(R_m)|}{|C_{R_m}(\Sigma S)|}\,.
\]
We have $C_{\ZZ_p} (\Sigma S) \subseteq K \cap \Sigma K \Sigma^{-1}$ and $C_{R_m}(\Sigma S) \subseteq (K \cap \Sigma K \Sigma^{-1}) \mod p^m$. Consider the equation $U \Sigma S U\T = \Sigma S$. Divide $U$ into blocks according to $m_k$ so that
\[
U = \begin{bmatrix}
U_{0,0} & U_{0,1} & U_{0,2} & \cdots & U_{0,t} \\
p U_{1,0} & U_{1,1} & U_{1,2}& \cdots & U_{1,t} \\
p^2 U_{2,0} & p U_{2,1} & U_{2,2} & \cdots & U_{2,t}\\
\vdots & \vdots &  & \ddots & \vdots \\
p^t U_{t,0} & p^{t-1} U_{t,1} & \cdots & p U_{t,t-1} & U_{t,t} \\
\end{bmatrix}
\]
with $U_{i,j} \in  M_{m_i\times m_j}(\ZZ_p)$, $U_{i,i} \in  \GL_{m_i}(\ZZ_p)$ and $S_{ii} = \mathbf{1}_{m_i}^{s_i}$. Rewrite $U \Sigma S U^{\top} = \Sigma S$ block-wise for a system of matrix equations with $(i,j)$-th block satisfying
% $$
% U_{i,i} S_{ii} U_{i,i}\T  + \sum_{k=0}^{i-1} p^{i-k} U_{i,k} S_{kk} U_{i,k}\T + \sum_{k=i+1}^{\infty} p^{k-i} U_{i,k} S_{kk} U_{i,k}\T= S_{ii} \mod p^{m-i};
% $$
\[
\sum_{k=i}^j p^j U_{i,k} S_{kk} U_{j,k}\T + \sum_{k=0}^{i-1} p^{i+j-k} U_{i,k} S_{kk} U_{j,k}\T + \sum_{k=j+1}^{\infty} p^{k} U_{i,k} S_{kk} U_{j,k}\T= \mathbf{0}\ \mod p^{m}\,.
\]
Divide out $(i,j)$-th block by $p^j$ and separate the diagonal and off-diagonal conditions:
\[\begin{cases}
U_{ii} S_{ii} U_{ii}\T  + \sum\limits_{k=0}^{i-1} p^{i-k} U_{ik} S_{kk} U_{ik}\T + \sum\limits_{k=i+1}^{\infty} p^{k-i} U_{ik} S_{kk} U_{ik}\T= S_{ii}\ \mod p^{m-i} &\text{ for $i \le t$}\,; \\
\\
\sum\limits_{k=i}^j U_{ik} S_{kk} U_{jk}\T + \sum\limits_{k=0}^{i-1} p^{i-k} U_{ik} S_{kk} U_{jk}\T + \sum\limits_{k=j+1}^{\infty} p^{k-j} U_{ik} S_{kk} U_{jk}\T= \mathbf{0}\  \mod p^{m-j} & \text{ for $i < j \le t$}\,.\end{cases}
\]
Solve this system mod $p$, where it reduces to
\[\begin{cases}
U_{ii} S_{ii} U_{ii}\T  = S_{ii} & \text{ for $i \le t$}\,; \\
\\
\sum\limits_{k=i}^j U_{ik} S_{kk} U_{jk}\T = \mathbf{0} & \text{ for $i < j \le t$}\,,\end{cases}
\]
and lift the solutions to $R_m$. Fix the lower blocks with $i > j$ arbitrarily, solve the diagonal equations with $U_{ii} \in \text{O}_{m_i}^{s_i}(\FF_p)$, then solve the remaining linear system for $j = i+1,\ldots, t$ consecutively (``diagonal by diagonal''). The blocks $U_{ik}$ with $i <k$ are determined uniquely from $U_{ik}$ with $ i \ge k$. It now remains to show that this system is ``well-defined'' i.e. its Jacobian matrix is full-rank. To avoid writing out each element-wise derivative separately, use matrix derivatives. The Jacobian of the system mod $p$ consists of the following matrix derivatives:
\begin{align*}
& B_{ii}^{ii} = \frac{\partial (U_{ii} S_{ii} U_{ii}\T)}{\partial U_{ii}} =  E\brackets{(U_{ii} S_{ii} \otimes \mathbf{1}_{m_i}^+) K + (\mathbf{1}_{m_i}^+ \otimes U_{ii}S_{ii})}  & \text{for $i \le t$}\,;\\
& B_{jj}^{ij}=  \frac{\partial (\sum_{k=i}^j U_{ik} S_{kk} U_{jk}\T)}{\partial U_{jj}}  = \frac{\partial (U_{ij} S_{jj} U_{jj}\T)}{\partial U_{jj}} = U_{ij} S_{jj} \otimes \mathbf{1}_{m_j}^+& \text{for $i < j \le t$}\,;\\
& B_{il}^{ij} = \frac{\partial (\sum_{k=i}^j U_{ik} S_{kk} U_{jk}\T)}{\partial U_{il}} = \frac{\partial  (U_{il} S_{ll} U_{jl}\T)}{\partial U_{il}} = S_{ll} U_{jl}\T \otimes \mathbf{1}_{m_i}^+  &\text{ for $l = i \ldots j$}\,.\end{align*}

Here $\otimes$ is the tensor product, $E$ and $K$ are the elimination and the commutation matrix. The diagonal derivatives $B_{ii}^{ii}$ are as in Lemma \ref{denslem}. The Jacobian matrix mod $p$ is block-lower-triangular.
% \[
% \mathbf{J}_{\mathbf{f}}(U) = \begin{bmatrix}
% B_{00}^{00}&0 &\cdots& \cdots & \cdots &\cdots & \cdots & \cdots & 0 \\
%   & B_{11}^{11}& 0 &\cdots & \cdots & \cdots & \cdots & \cdots & 0 \\
%   &   &  \ddots &  &   &   &   &    &  \vdots\\
%   &  &  &B_{tt}^{tt} & 0 & \cdots & \cdots & \cdots & 0 \\
%   &  &&  & B_{01}^{01} & 0 & \cdots & \cdots & 0 \\
%   &  &&   &   & \ddots  &   &  & \vdots \\
%   & &&  &  &   & B_{t-2, t-1}^{t-2, t-1} & 0 & 0 \\
%   &  && &   &   &   & B_{t-2, t}^{t-2, t} & 0 \\
%   &  &&  &   &   &   &   & B_{t-1,t}^{t-1, t}\\
% \end{bmatrix}\,.
% \]
On the diagonal blocks $B_{ii}^{ii}$ for $i = 0,\ldots, t$ come first, then $B_{ij}^{ij}$ for $i = 0,\ldots,t$ and $j = i+1,\ldots, t$. This corresponds to how we solved the system mod $p$. Blocks $B_{ij}^{ij} = S_{jj} U_{jj}\T \otimes \mathbf{1}_{m_i}^+$ are full-rank as $\det B_{ij}^{ij} = (\det S_{jj} U_{jj})^{m_i} \in \ZZ_p^{\times}$, and so are $B_{ii}^{ii}$ by Lemma \ref{denslem}. Thus the Jacobian has full row rank. 
% Substituting $U = \mathbf{1}$ to this system we get
% $$
% \frac{\partial (\sum_{k=i}^j U_{ik} S_{kk} U_{jk}\T)}{\partial U_{ik}}\bigg\vert_{\mathbf{1}} = 0 \text{ for $k = i \ldots j-1$,} \quad
% \frac{\partial (\sum_{k=i}^j U_{ik} S_{kk} U_{jk}\T)}{\partial U_{ij}}\bigg\vert_{\mathbf{1}} = S_{kk} \otimes \mathbf{1};
% $$
% and 
% $$
% \frac{\partial (\sum_{k=i}^j U_{ik} S_{kk} U_{jk}\T)}{\partial U_{jj}}\bigg\vert_{\mathbf{1}}  = 0.
% $$
% Diagonal blocks have rank $l_i(l_i+1)/2$ by Lemma \ref{denslem}, and the off-diagonal contribute $l_il_j$, hence
% $$
% \rk \mathbf{J} = \sum_i \frac{l_i(l_i+1)}{2} + \sum_{i < j} l_i l_j = \sum_i \frac{l_i^2}{2} + \frac{n}{2} + \frac{n^2 - \sum_i l_i^2}{2} = \frac{n(n+1)}{2}.
% $$

% Each block $U_{ij}$ with $i < j$ is determined mod $p^{m-j}$. 
Blocks $U_{ii}$ have $|\text{O}_{m_i}^{s_i}(\ZZ/p^{m-i} \ZZ)| p^{im_i^2} = \alpha_{m_i}^{s_i} (p^{m-i})^{m_i(m_i-1)/2} p^{im_i^2}$ lifts mod $p^{m-i}$. For $i > j$ blocks $U_{ij} \in \M_{m_i\times m_j}(p^{m} \ZZ/p^{i-j} \ZZ)$ with $p^{(m-(i-j))m_im_j}$ options mod $p^{m-(i-j)}$, and the upper blocks are determined uniquely mod $p^{m-j}$ by Hensel's lemma. Hence $|C_{R_m}(\Sigma S)|$ equals
\begin{equation*}
\prod_{i=0}^{\infty} \alpha_{m_i}^{s_i} p^{(m-i)m_i(m_i-1)/2 + im_i^2} 
\prod_{i < j} p^{(m-(j-i))m_im_j + jm_im_j} =  p^{mn(n-1)/2}\prod_{i=0}^{\infty} \alpha_{m_i}^{s_i}
\prod_{j =1}^n p^{k_j(n-j+1)}.
\end{equation*}
% If all $k_i = t$, we have $|C_{R_m}(p^t \mathbf{1}_n^{s_t})| = \alpha_{n}^{r} p^{mn(n-1)/2} p^{tn(n+1)/2}$. If all $k_i$ are distinct, $|C_{R_m}(\Sigma S)| = \prod_{i=0}^{\infty} \alpha_{1}^{s_{k_i}} p^{mn(n-1)/2}
% \prod_{i =1}^n p^{k_i(n-i+1)}$. 
% In the general case we have
% \[
% |C_{R_m}(\Sigma S)| =p^{mn(n-1)/2} \prod_{i=0}^{\infty} \alpha_{m_i}^{s_i} 
% \prod_{j =1}^n p^{k_j(n-j+1)}\,,
% \]
This then implies the statement of the theorem.
\end{proof}

\begin{cor}\label{symel}
Let $\Sigma = \diag(p^{k_1},\ldots,p^{k_n})$ be a matrix of elementary divisors with multiplicities $m_k= \#\{i:k_i=k\}$, then
\[
\mu (K \Sigma K \cap S_n(\ZZ_p)) = \frac{\pi_n}{\beta_{m_0} \cdots \beta_{m_k} \cdots} \prod_{i=1}^{n} p^{-k_j(n-j+1)} ,
\]
where $(\beta_t)^{-1} =  (\alpha_t^+)^{-1} + (\alpha_t^-)^{-1} = \prod_{j=1}^{\lfloor\frac{t}{2}\rfloor} (1-p^{-2j})$.
\end{cor}
% \begin{proof}
% Sum over all admissible signature matrices $S$ with
% $$
% \PP\{X \in K \Sigma K\} = \sum_{S} \PP\{X \in K \Sigma S K^\top\} = \prod_{j=1}^{n} p^{-k_j(n-j+1)} \sum_{s_i \in \{\pm\}, m_i>0}  \frac{\pi_n}{\alpha_{m_0}^{s_0} \alpha_{m_1}^{s_1}\cdots}\,,
% $$
% which implies the statement immediately.
% \end{proof}

Now let us give a brief discussion regarding the connection between the distributions in question and symmetric functions. Recall that the density of elementary divisors (\ref{gensymf}) of general random matrices over $\ZZ_p$ can be written in terms of symmetric functions. Here we similarly have
\begin{align*}
\mu(K \Sigma S K\T) = 
\frac{\pi_{m_0} Q_{\mathbf{k}} (p^{-\mathbf{n}};p^{-1})}{\alpha_{m_0}^{s_0}\cdots\alpha_{m_k}^{s_k}\cdots}\,,
\end{align*}
where $p^{-\mathbf{n}} = (p^{-1},p^{-2},\ldots, p^{-n})$ and $Q_{\mathbf{k}}(\mathbf{x};t)$ is the Hall-Littlewood $Q$ polynomial \cite{macdonald1998symmetric}
\begin{align*}
& Q_{\mathbf{k}}(x_1,\ldots,x_n;t) = \frac{(1-t)^n}{\pi_{m_0}} \sum_{\sigma \in S_n} \sigma\brackets{\mathbf{x}^{\mathbf{k}} \prod_{i <j} \frac{x_i - tx_j}{x_i-x_j}}\,;\\ 
& Q_{\mathbf{k}}(p^{-\mathbf{n}};p^{-1}) = \frac{\pi_n}{\pi_{m_0}} p^{-\sum_i (n-i+1)k_i}\,.
\end{align*}
Let $d_{\mathbf{k}}(t) := \prod_{i \ge 1} \prod_{j=1}^{\lfloor \frac{m_i}{2}\rfloor} (1-t^{2j})$, and rewrite the density in Corollary \ref{symel} as
\[
\mu\brackets{K\Sigma K} = \frac{\pi_{m_0}}{\beta_{m_0}} \frac{Q_{\mathbf{k}} (p^{-\mathbf{n}};p^{-1})}{ d_{\mathbf{k}}(p^{-1})} = \mu(S_{m_0}(\ZZ_p) \cap \GL_{m_0}(\ZZ_p)) \frac{Q_{\mathbf{k}} (p^{-\mathbf{n}};p^{-1})}{ d_{\mathbf{k}}(p^{-1})}\,.
\]
This form is helpful when considering the limit $n \rightarrow \infty$.

 Let $\lambda = (\lambda_1,\ldots, \lambda_l, 0,\ldots)$ be a partition of length $\ell(\lambda) = l$ with $\lambda_1 \ge \lambda_2\ge \ldots \ge \lambda_l > 0$ and $m_j(\lambda) = \#\{i:\lambda_i = j\}$. Partitions are defined up to an arbitrary string of zeros at the end. Let $f_n(\lambda)$ be the probability that the exponents of elementary divisors are equivalent to $\lambda$ as a partition, then
\begin{equation}\label{fnlambda}
f_n(\lambda) = \begin{cases} \frac{\pi_n}{\beta_{n-\ell(\lambda)}d_{\lambda}(p^{-1})} \prod_{i = 1}^{\ell(\lambda)} p^{-i\lambda_i} =  \frac{\pi_{n-\ell(\lambda)}}{\beta_{n-\ell(\lambda)}} \frac{Q_{\mathbf{\lambda}} (p^{-\mathbf{n}};p^{-1})}{ d_{\mathbf{\lambda}}(p^{-1})}\,, & \ell(\lambda) \le n\,;\\
0,& \text{otherwise}\,.
\end{cases}
\end{equation}
\begin{prop} Let $f_n(\lambda)$ be the distribution on partitions induced by the elementary divisors. Then the limiting distribution on partitions as $n \rightarrow\infty$ exists and is equal to 
\[
f(\lambda) := \lim_{n\rightarrow\infty} f_n(\lambda) = \frac{\pi_\infty}{\beta_{\infty}d_{\lambda}(p^{-1})} p^{-\sum i\lambda_i}\,.
\]
\end{prop}
\begin{proof} It is clear that the point-wise limit of $f_n(\lambda)$ exists and is positive, so it remains to check $\sum_{\lambda} f(\lambda) = 1$. Function $f_n(\lambda)$ is bounded from above by an integrable function
\[
f_n(\lambda)\le \frac{1}{\beta_{\infty} d_{\lambda}(1/p)} p^{-\sum_i i\lambda_i} \le \frac{1}{\beta_{\infty}(1-1/p)^{\ell(\lambda)}} p^{-\sum_i i\lambda_i} 
%  \le \frac{1}{\beta_{\infty}(1-1/p)^{\ell(\lambda)}} p^{-\frac{\ell(\lambda)(\ell(\lambda)+1)}{2}} 
=: g(\lambda)\,.
\]
and
\[
\sum_\lambda g(\lambda) \ll \sum_{l=0}^{\infty} (1-1/p)^{-l} \sum_{\lambda:\ell(\lambda) = l} \prod_{i=1}^{l}p^{-i\lambda_i} \ll \sum_{l=0}^{\infty} (1-1/p)^{-l} p^{-\frac{l(l+1)}{2}} < +\infty\,.
\]
By dominated convergence $\sum_{\lambda} f(\lambda) = \lim_{n\rightarrow\infty} \sum_{\lambda} f_n(\lambda) =  1$, thus $f(\lambda)$ defines a distribution. 
% We could have also used that $Q_{\lambda}$ and $d_{\lambda}$ satisfy the identity (Chapter 3 \cite{macdonald1998symmetric})
% \[
% \sum_{\mathbf{\lambda}} \frac{Q_{\mathbf{\lambda}}(\mathbf{x};t)}{d_{\mathbf{\lambda}}(t)} = \prod_i \frac{1-tx_i}{1-x_i} \prod_{i<j} \frac{1-tx_ix_j}{1-x_ix_j}
% \]
% with
% \[
% \sum_{\lambda} \frac{p^{-\sum_i i\lambda_i}}{d_{\lambda}(p^{-1})}  =  \sum_{\lambda} \frac{Q_{\mathbf{\lambda}} (p^{-1},p^{-2},\ldots;p^{-1})}{ d_{\mathbf{\lambda}}(p^{-1})} = \frac{\beta_{\infty}}{\pi_{\infty}}\,,
% \]
% where $Q_{\lambda}(p^{-1},p^{-2},\ldots;p^{-1}) := \lim_{n \rightarrow\infty} Q_{\mathbf{\lambda}} (p^{-\mathbf{n}};p^{-1}) = p^{-\sum i\lambda_i}$. 
\end{proof}

\begin{remark}\normalfont
Note that $f_n(\lambda)$ is a rational function in $t = p^{-1}$, so $\sum_{\lambda} f_n(\lambda) =1
$ is a rational identity for all $t \in (0,1)$, and more generally can be seen as a formal identity. The proof by dominated convergence works for all $t \in (0,1)$ and can be interpreted as a probabilistic proof of the identity
\[
\sum_{\lambda} \frac{Q_{\mathbf{\lambda}} (t,t^2,\ldots;t)}{ d_{\mathbf{\lambda}}(t)} = \prod_{j=0}^{\infty} (1-t^{2j+1})^{-1}\,.
\]
\end{remark}

\subsection{$\QQ_p$-equivalence}
Let $p$ be an odd prime, and let $Q$ be a random quadratic form in $n$ variables with independent coefficients distributed according to the Haar measure on $\ZZ_p$. This corresponds to the distribution on $S_n(\ZZ_p)$ we considered earlier. Let $a \in \ZZ_p/(\ZZ_p)^2$ and $b \in\{\pm 1\}$, and denote the probability that $Q$ has invariants $d(Q) = a$ and $c(Q) = b $ by $\rho_n(a,b)$ for $n \in \NN$. When $n=0$ set $\rho_0(a,b) = \mathbf{1}\{a=1,b=1\}$.

Recall that
\[
\chi(x) = \begin{cases}+1\,, & \text{if $x \in \ZZ_p^{\times} \cap (\ZZ_p)^2$}\,;\\
-1\,,& \text{if $x \in \ZZ_p^{\times} \setminus (\ZZ_p)^2$}\,;\\
0\,,& \text{if $x \in p\ZZ_p$}\,.
\end{cases}
\]
In particular, $\chi(1) = +1$ and $\chi(r) = -1$, where $r \in \ZZ_p^{\times}$ is a fixed non-square in $\ZZ_p$. We will write $\chi(1) = +$ and $\chi(r) = -$ to match the notation of $\alpha_n^{\pm}$. Denote $\chi(-1)$ by $\varepsilon$. 

We will also need the following analogue of Corollary \ref{genrank} for alternating matrices. Alternating, or skew-symmetric, matrices are matrices satisfying $A^\top = -A$. In particular, this implies that they have zero diagonal.
\begin{lem}[Carlitz \cite{carlitz1954representations2}] Let $A$ be a random alternating $n \times n$ matrix with independent entries distributed uniformly on $\FF_q$ (of odd characteristic). Then for all $r$ such that $n-r$ is even
\[
\PP\{\rk\, A = n - r\} = q^{-r(r-1)/2} \frac{\pi_n(q)}{\pi_r(q)\beta_{n-r}(q)}\,.
\]
If $n-r$ is odd, $\PP\{\rk\, A = n - r\} = 0$.
\label{altdistr}\end{lem}

Condition on $m_0$ and $s_0$ to obtain the recursive expression for $\rho_n$:
\[
\rho_n(a,b) = \sum_{\substack{l=0\\ d \in \{1,r\}}}^n \PP_n\{m_0 = n-l, s_0 = \chi(d)\} \rho_l(a^\prime(l,d),b^\prime(l,d))\,,
\]
where $a^\prime(l,d) = adp^l$ and $b^\prime(l,d) = b \varepsilon^{\frac{l(l-1)}{2}} \left<a,d\right>\left<p,ad\right>^{l-1}$. This corresponds to separating out the non-zero mod $p$ part, factoring out the rest by $p$ and inferring the invariants of the remaining quadratic form in $n-m_0$ variables. By Theorem \ref{symdensity}
\[
\PP_n\{m_0 = n-l, s_0 = s\} = p^{-\frac{l(l+1)}{2}} \frac{\pi_n}{ \pi_l \alpha_{n-l}^s}\,,
\]
hence
% \begin{equation}\label{recurr}
%  \begin{aligned}
% \rho_n(a,b) = 
% p^{-\frac{n(n+1)}{2}} \rho_n\big(ap^n, b \varepsilon^{\frac{n(n-1)}{2}} \left<p,a\right>^{n-1}\big) + \sum_{l=0}^{n-1} p^{-\frac{l(l+1)}{2}} \frac{\pi_n}{ \pi_l \alpha_{n-l}^+} \rho_l\big(ap^l, b \varepsilon^{\frac{l(l-1)}{2}} \left<p,a\right>^{l-1}\big)\,+ \\ \sum_{l=0}^{n-1} p^{-\frac{l(l+1)}{2}} \frac{\pi_n}{ \pi_l \alpha_{n-l}^-} \rho_l\big(arp^l, b \varepsilon^{\frac{l(l-1)}{2}} \left<a,r\right>\left<p,ar\right>^{l-1}\big)\,.
% \end{aligned}  
% \end{equation}
\begin{equation}\label{recurr}
 \begin{aligned}
\rho_n(a,b) = 
p^{-\frac{n(n+1)}{2}} \rho_n(a^\prime(n,1),b^\prime(n,1)) + \sum_{l=0}^{n-1} p^{-\frac{l(l+1)}{2}} \frac{\pi_n}{ \pi_l \alpha_{n-l}^+} \rho_l(a^\prime(l,1),b^\prime(l,1))\,+ \\ \sum_{l=0}^{n-1} p^{-\frac{l(l+1)}{2}} \frac{\pi_n}{ \pi_l \alpha_{n-l}^-} \rho_l(a^\prime(l,r),b^\prime(l,r))\,.
\end{aligned}  
\end{equation}
First, we will compute
\begin{align*}
     \sigma_n(a):= &\ \rho_n(a,1) + \rho_n(a,-1) = \PP_n\{d(Q) = a\}\,,\\
     \Delta_n(a) := &\ \rho_n(a,1) - \rho_n(a,-1)\,.
\end{align*}
Then $\rho_n(a,b) = (\sigma_n(a) + b \Delta_n(a))/2$.
% \begin{remark}
% We will use in computations that $\beta_{2k} = \beta_{2k+1}$ and $\alpha_{2k+1}^+ = \alpha_{2k+1}^- = 2\beta_{2k+1}$.
% \end{remark}

Let us start with the distribution of the discriminant $\sigma_n(a)$.

\begin{lem}\label{sigmalem} Let $n \in \ZZ^{\ge 0}$, set $t = 1/p$ and $s = \chi(a)$. If $n$ is even, we have
\[
\sigma_n(a)  = \begin{cases}
\frac{(1+ s(\varepsilon t)^{n/2}) ( 1 - s \varepsilon^{n/2} t^{n/2+2})}{2(1+t)(1-t^{n+1})}\,, & \text{if $a \in \{1,r\}$}\,;\\
\\
\frac{t(1-t^{n})}{2(1+t)(1-t^{n+1})}\,, & \text{if $a \in \{p,pr\}$}\,;
\end{cases}
\]
and in $n$ is odd,
\[
\sigma_n(a)  = \begin{cases}
\frac{1}{2\brackets{1 + t}}\,, & \text{if $a \in \{1,r\}$}\,;\\
\\
\frac{t}{2\brackets{1 + t}}\,, & \text{if $a \in \{p,pr\}$}\,.
\end{cases}
\]
\end{lem}
% \begin{lem}\label{sigmalem} Let $n \in \ZZ^{\ge 0}$, set $t = 1/p$ and $s = \chi(a)$. Then we have
% \[
% \sigma_n(a)  = \begin{cases}
% \frac{(1+ s(\varepsilon t)^{n/2}) ( 1 - s \varepsilon^{n/2} t^{n/2+2})}{2(1+t)(1-t^{n+1})}\,, & \text{if $n$ is even and $a \in \{1,r\}$}\,;\\
% \\
% \frac{t(1-t^{n})}{2(1+t)(1-t^{n+1})}\,, & \text{if $n$ is even and $a \in \{p,pr\}$}\,;\\
% \\
% \frac{1}{2\brackets{1 + t}}\,, & \text{if $n$ is odd and $a \in \{1,r\}$}\,;\\
% \\
% \frac{t}{2\brackets{1 + t}}\,, & \text{if $n$ is odd and $a \in \{p,pr\}$}\,.
% \end{cases}
% \]
% \end{lem}

\begin{proof}
First, condition on the multiplicity $m_0$ and the signature $s_0$ and write the recurrence for $\sigma_n(a)$:
\begin{equation}\label{sigmarec}
   \begin{aligned}
   \sigma_n(a) = \sum_{l=0}^{n-1} p^{-\frac{l(l+1)}{2}} \brackets{\frac{\pi_n}{ \pi_l \alpha_{n-l}^+} \sigma_l(ap^l) \ + 
\frac{\pi_n}{ \pi_l \alpha_{n-l}^-} \sigma_l(arp^l)} + p^{-\frac{n(n+1)}{2}} \sigma_n(ap^n)\,.
   \end{aligned} 
\end{equation}
% It will be helpful to keep in mind that $\beta_{2k} = \beta_{2k+1}$ and $\alpha_{2k+1}^+ = \alpha_{2k+1}^- = 2\beta_{2k+1}$.
To reduce the number of cases, observe that if $\text{val}_p(\det Q)$ is odd, or if $n$ is odd, then $Q$ has an elementary divisor of odd multiplicity $l$. Then we can choose $s_l \in \{\pm\}$ with equal probability as $\alpha_{l}^+ = \alpha_{l}^-$, hence $\sigma_n(1) = \sigma_n(p)$ for odd $n$, and $\sigma_n(p) = \sigma_n(rp)$ for all $n$. 

The statement holds for $n = 0$ and is also easy to verify for small $n$. Suppose it holds for all $l < n$ and substitute coefficients with $l < n$ into \eqref{sigmarec}. For even $n$, and $a = 1$ or $r$ with $s = \chi(a)$
\begin{align*}
\sigma_n(a) = t^{\frac{n(n+1)}{2}}\sigma_n(a) +  \sum_{\substack{l \text{ odd}\\ l < n}}  t^{\frac{l(l+1)}{2}} \frac{t\pi_n}{ 2(1+ t)\pi_l \beta_{n-l}}  +
\sum_{\substack{l \text{ even}\\ l < n}} t^{\frac{l(l+1)}{2}} \frac{\pi_n}{ \pi_{l+1} \beta_{n-l}} \frac{ 1 - t^{l+2} + s (1-t^2)(\varepsilon t)^{\frac{n}{2}}}{2(1 + t)}\, = \\
 t^{\frac{n(n+1)}{2}}\sigma_n(a) +  \frac{ 1 - t^{n+2} + s (1-t^2)(\varepsilon t)^{\frac{n}{2}}}{2(1 + t)(1 -t^{n+1})} \sum_{\substack{l \text{ even}\\ l < n}} t^{\frac{l(l+1)}{2}} \frac{\pi_{n+1}}{ \pi_{l+1} \beta_{n-l}}\,.
\end{align*}
For $a = p$ or $pr$ we have
\begin{align*}
\sigma_n(a) = t^{\frac{n(n+1)}{2}}\sigma_n(a) + \sum_{\substack{l \text{ odd}\\ l < n}}  t^{\frac{l(l+1)}{2}} \frac{\pi_n}{ 2(1+ t)\pi_l \beta_{n-l}} + 
\sum_{\substack{l \text{ even}\\ l < n}} t^{\frac{l(l+1)}{2}} \frac{\pi_n}{ \pi_{l+1} \beta_{n-l}} \frac{t( 1 - t^{n})}{2(1 + t)}\,= \\
t^{\frac{n(n+1)}{2}}\sigma_n(a) +  \frac{ 1 - t^{n}}{2(1 + t)(1 -t^{n+1})} \sum_{\substack{l \text{ even}\\ l < n}} t^{\frac{l(l+1)}{2}} \frac{\pi_{n+1}}{ \pi_{l+1} \beta_{n-l}}\,.
\end{align*}
The sum over even $l$ equals $1 - p^{-n(n+1)/2}$ as it is the sum over the distribution of rank of $(n+1) \times (n+1)$ alternating matrices over $\FF_p$ as in Lemma \ref{altdistr} without the rank zero term. This implies the statement of the lemma for even $n$. For odd $n$
\begin{align*}
\sigma_n(1) = t^{\frac{n(n+1)}{2}}\sigma_n(p) + \sum_{\substack{l \text{ odd}\\ l < n}}  t^{\frac{l(l+1)}{2}} \frac{t\pi_n}{ 2(1+ t)\pi_l \beta_{n-l}} + \sum_{\substack{l \text{ even}\\ l < n}} t^{\frac{l(l+1)}{2}} \frac{\pi_n}{ \pi_{l+1} \beta_{n-l}} \frac{ 1 - t^{l+2}}{2(1 + t)}\, = \\
t^{\frac{n(n+1)}{2}}\sigma_n(p) +  \frac{1}{2(1 + t)} \sum_{\substack{l \text{ even}\\ l < n}} t^{\frac{l(l+1)}{2}} \frac{\pi_{n}}{ \pi_{l+1} \beta_{n-l}} - t^{\frac{n(n+1)}{2}}\frac{t}{2(1+t)}\,;\\
\sigma_n(p) = t^{\frac{n(n+1)}{2}}\sigma_n(1) + \sum_{\substack{l \text{ odd}\\ l < n}}  t^{\frac{l(l+1)}{2}} \frac{\pi_n}{ 2(1+ t)\pi_l \beta_{n-l}} + \sum_{\substack{l \text{ even}\\ l < n}} t^{\frac{l(l+1)}{2}} \frac{\pi_n}{ \pi_{l+1} \beta_{n-l}} \frac{ t(1 - t^{l})}{2(1 + t)}\,= \\
t^{\frac{n(n+1)}{2}}\sigma_n(1) +  \frac{t}{2(1 + t)} \sum_{\substack{l \text{ even}\\ l < n}} t^{\frac{l(l+1)}{2}} \frac{\pi_{n}}{ \pi_{l+1} \beta_{n-l}} - t^{\frac{n(n+1)}{2}}\frac{1}{2(1+t)}\,.
\end{align*}
The sum over the even indices is the sum over the distribution of rank of $n \times n$ alternating matrices over $\FF_p$, hence equals $1$. Solving this system we obtain $\sigma_n(1)=\sigma_n(r)$ and $\sigma_n(p) = \sigma_n(rp)$.
\end{proof}
Now let us compute the discrepancy between the distribution of the discriminant of quadratic forms with different Hasse invariants.

\begin{lem}\label{deltalem}
Let $n \in \ZZ^{\ge 0}$, and set $t = 1/p$. Then
\[
\Delta_n(a)  = \begin{cases}
\frac{\pi_n}{\beta_{n+1}\alpha_n^s}\,, & \text{if $a = 1$ or $r$, and $s = \chi(a)$}\,;\\
\\
\varepsilon^{\frac{n-1}{2}} t^{\frac{n+1}{2}} \frac{\pi_n}{2\beta_{n+1}\beta_{n-1}}\,, & \text{if $a = p$ or $pr$, and $n$ is odd}\,;\\
\\
0\,, & \text{if $a = p$ or $pr$, and $n$ is even}\,.
\end{cases}
\]
\end{lem}

\begin{proof} Similarly to the previous lemma, prove the statement by induction. Conditioning on the multiplicity $m_0$, write the recurrence for $\Delta_n(a)$ using (\ref{recurr}):
\begin{equation}\label{deltarec}
\begin{aligned}
\Delta_n(a) = \sum_{l=0}^{n} p^{-\frac{l(l+1)}{2}} \varepsilon^{\frac{l(l-1)}{2}} \frac{\pi_n}{ \pi_l \alpha_{n-l}^+}   \left<p,a\right>^{l-1} \Delta_l(ap^l) + \\
 \left<a,r\right>\sum_{l=0}^{n-1} p^{-\frac{l(l+1)}{2}} \varepsilon^{\frac{l(l-1)}{2}} \frac{\pi_n}{ \pi_l \alpha_{n-l}^-} \left<p,ar\right>^{l-1} \Delta_l(arp^l)\,.
\end{aligned}
\end{equation}
    The statement holds for $n = 0$. It is also easy to check for $n\le 2$. Suppose the statement holds for all $l < n$. Split this sum into the sum over odd and even indices, and substitute $\Delta_l(\cdot)$ for $l < n$ into (\ref{deltarec}). Consider first $a = 1$ or $r$ with $s = \chi(a)$, then
\begin{align*}
\Delta_n(a)- t^{\frac{n(n+1)}{2}} \varepsilon^{\frac{n(n-1)}{2}} s^{n-1} \Delta_n(ap^n)\ = \sum_{\substack{l \text{ odd}\\ l < n}} t^{\frac{l(l+1)}{2}} \varepsilon^{\frac{l(l-1)}{2}} \frac{\pi_n}{\pi_l} \bigg(\frac{\Delta_l(ap)}{\alpha_{n-l}^+}  + \frac{\Delta_l(apr)}{\alpha_{n-l}^-} \bigg)\, +\\ 
+\ s \sum_{\substack{l \text{ even}\\ l < n}} t^{\frac{l(l+1)}{2}} \varepsilon^{\frac{l(l-1)}{2}} \frac{\pi_n}{\pi_l} \bigg(\frac{\Delta_l(a)}{\alpha_{n-l}^+} - \frac{\Delta_l(ar)}{\alpha_{n-l}^-}  \bigg)\,.
\end{align*}
When $n$ is even, we have
\begin{align*}
% \varepsilon^{\frac{n(n-1)}{2}} p^{-\frac{n(n+1)}{2}} \Delta_n (1) +
\Delta_n(a) (1 - st^{\frac{n(n+1)}{2}} \varepsilon^{\frac{n(n-1)}{2}}) = \sum_{\substack{l \text{ odd}\\ l < n}} t^{\frac{(l+1)^2}{2}}  \frac{\pi_n}{ 2\beta_l \beta_{l+1} \beta_{n-l}}\ +\ \sum_{\substack{l \text{ even}\\ l < n}} t^{\frac{l(l+1)}{2}} \frac{\pi_n}{ 2\beta_l \beta_{l+1} \beta_{n-l}} \brackets{t^{\frac{l}{2}} +s \varepsilon^{\frac{n}{2}} t^{\frac{n-l}{2}}}\,=\\ 
\frac{\pi_n}{\beta_n \alpha_n^s} \bigg(\beta_n + \sum_{\substack{l \text{ odd}\\  l < n}}  t^{\frac{(l+1)^2}{2}} \frac{\beta_n^2}{ \beta_{l+1}^2 \beta_{n-l-1}} - s\,t^{\frac{n(n+1)}{2}} \varepsilon^{\frac{n(n-1)}{2}}\bigg) =\\
\frac{\pi_n}{\beta_n \alpha_n^s}  \bigg(\sum_{j=0}^{n/2} t^{2j^2} \frac{\beta_n^2}{\beta_{2j}^2 \beta_{n-2j}} - s\,t^{\frac{n(n+1)}{2}} \varepsilon^{\frac{n(n-1)}{2}}\bigg)\,.
\end{align*}
The sum over $j$ is the sum over the distribution of rank of general square $n/2 \times n/2$ matrices over $\FF_{p^2}$, and by Corollary \ref{genrank} it equals $1$. Hence the statement holds for $\Delta_n(1)$ and $\Delta_n(r)$. If $n$ is odd, 
\begin{align*}
    \Delta_n(1) - \varepsilon^{\frac{n(n-1)}{2}} t^{\frac{n(n+1)}{2}} \Delta_n (p) = \sum_{\substack{l \text{ odd}\\ l < n}} t^{\frac{(l+1)^2}{2}}  \frac{\pi_n}{ 2\beta_l \beta_{l+1} \beta_{n-l}} + \sum_{\substack{l \text{ even}\\ l < n}} t^{\frac{l(l+2)}{2}} \frac{\pi_n}{ 2\beta_l^2 \beta_{n-l}} = \\
     \frac{\pi_{n}}{2\beta_{n+1}\beta_{n-1}}\brackets{\beta_{n+1} +  \sum_{\substack{l \text{ odd}\\ l < n}} t^{\frac{(l+1)^2}{2}} \frac{\beta_{n+1}^2}{ \beta_{l+1}^2 \beta_{n-l}}} = \frac{\pi_{n}}{2\beta_{n+1}\beta_{n-1}} \sum_{j=0}^{(n-1)/2} t^{2j^2} \frac{\beta_{n+1}^2}{ \beta_{2j}^2 \beta_{n-2j+1}}.
\end{align*}
The sum over $j$ is the sum over the distribution of rank of $(n+1)/2 \times (n+1)/2$ matrices over $\FF_{p^2}$ except rank zero, hence it equals $1 - p^{-(n+1)^2/2}$. 

Further, $\Delta_n(p) = \Delta_n(pr)$ satisfies the recurrence
\begin{align*}
\Delta_n(p) = \varepsilon^{\frac{(n-1)(n+2)}{2}} t^{\frac{n(n+1)}{2}} \Delta_n (p^{n+1}) \ + \sum_{\substack{l \text{ odd}\\ l < n}} t^{\frac{l(l+1)}{2}} \varepsilon^{\frac{l(l-1)}{2}} \frac{\pi_n}{\pi_l} \brackets{\frac{\Delta_l(1)}{\alpha_{n-l}^+}\ - \frac{\Delta_l(r)}{\alpha_{n-l}^-} }\, +\\ 
+\ \varepsilon \sum_{\substack{l \text{ even}\\ l < n}} t^{\frac{l(l+1)}{2}} \varepsilon^{\frac{l(l-1)}{2}} \frac{\pi_n}{\pi_l} \brackets{\frac{\Delta_l(p)}{\alpha_{n-l}^+} \ + \frac{\Delta_l(pr)}{\alpha_{n-l}^-}  }\,,
\end{align*}
where we conveniently have $\Delta_l(p) = \Delta_l(pr) = 0$ for even $l < n$. If $n$ is even, the sum over odd $l$ is also zero as $n-l$ is odd and $\alpha_{n-l}^+=\alpha_{n-l}^-$. Hence $\Delta_n(a) = 0$. For odd $n$
\begin{align*}
    \Delta_n(p) - \varepsilon^{\frac{n(n-1)}{2}} t^{\frac{n(n+1)}{2}} \Delta_n (1) =
    % \sum_{\substack{l \text{ odd}\\ l < n}} p^{-\frac{l(l+1)}{2}}  \frac{\pi_n}{ 2\beta_l \beta_{l+1} \beta_{n-l}} \varepsilon^{\frac{l(l-1)}{2}} \Big(\frac{\varepsilon}{p}\Big)^{\frac{n-l}{2}} = \\
     \varepsilon^{\frac{n-1}{2}} t^{\frac{n+1}{2}} \frac{\pi_{n}}{2\beta_{n+1}\beta_{n-1}} \sum_{\substack{l \text{ odd}\\ l < n}} t^{\frac{(l+1)(l-1)}{2}} \frac{\beta_{n+1}\beta_{n-1}}{ \beta_{l-1} \beta_{l+1} \beta_{n-l}} =\\  \varepsilon^{\frac{n-1}{2}} t^{\frac{n+1}{2}} \frac{\pi_{n}}{2\beta_{n+1}\beta_{n-1}} \sum_{j=0}^{(n-3)/2} t^{2j(j+1)} \frac{\beta_{n+1}\beta_{n-1}}{ \beta_{2j} \beta_{2j+2} \beta_{n-1-2j}}\,.
\end{align*}
The sum over $j$ is the sum over the distribution of rank for $(n+1)/2 \times (n-1)/2$ matrices over $\FF_{p^2}$ as in Corollary \ref{genrank} except the last term corresponding to rank $0$, hence it equals $1 - p^{-(n^2-1)/2}$. Solving this simple system we see that the statement holds for odd $n$.
\end{proof}
Combining Lemma \ref{sigmalem} and \ref{deltalem}, we obtain the following theorem.

\begin{thm} \label{qfclass}Let $n \in \ZZ^{\ge 1}$, and set $t=1/p$. Then for odd $n$
\[
\rho_n(a,b) = \begin{cases}
\frac{1}{4(1+t)} + b \frac{\pi_n}{4\beta_{n+1}\beta_{n-1}}\,, & \text{if $a \in \{1,r\}$}\,;\\
\\
\frac{t}{4(1+t)} + b \varepsilon^{\frac{n-1}{2}} t^{\frac{n+1}{2}} \frac{\pi_n}{4\beta_{n+1}\beta_{n-1}}\,, & \text{if $a \in \{p,pr\}$}\,;
\end{cases}
\]
and for even $n$ and $s = \chi(a)$
\[
\rho_n(a,b) = \begin{cases}
\frac{(1+ s(\varepsilon t)^{n/2}) ( 1 - s \varepsilon^{n/2} t^{n/2+2})}{4(1+t)(1-t^{n+1})} + b \frac{\pi_n}{2\beta_{n}\alpha_{n}^s}\,, & \text{if $a \in \{1,r\}$}\,;\\
\\
\frac{t(1-t^n)}{4(1+t)(1-t^{n+1})}\,, & \text{if $a\in\{p,pr\}$}\,.
\end{cases}
\]
\end{thm}
This also gives us the asymptotic distribution of quadratic forms independent of parity.
\begin{cor}
As $n \to \infty$ 
\[
\rho(a,b) := \lim_{n\rightarrow\infty} \rho_n(a,b) = \begin{cases} 
\frac{1}{4(1+t)} + b \frac{\pi_{\infty}}{4\beta_{\infty}^2}\,, & \text{if $a \in \{1,r\}$}\,;\\
\frac{t}{4(1+t)}\,, & \text{if $a \in \{p,pr\}$}\,.
\end{cases}
\]
\end{cor}
We expect that this distribution should be universal for random quadratic forms with independent coefficients from a large class of distributions not necessarily uniform.

Now let us give a few examples of problems where the distributions we derived come in handy.

% We have
% $$
% \PP_1\{a,b\}  = \begin{cases}
% \frac{1}{2(1+1/p)}, & \text{if $a = 1$ or $r$, $b=1$};\\
% \frac{1}{2p(1+1/p)}, & \text{if $a = p$ or $pr$, $b=1$};\\
% 0, & \text{otherwise.}
% \end{cases}
% $$
% For $n=2$ we compute that 
% $$
% \PP_2\{a,b\}  = \begin{cases}
% \frac{1}{2}, & \text{if $a = \varepsilon$, $b=1$};\\
% \frac{1-1/p}{2(1+1/p)(1-1/p^3)}, & \text{if $a = -\varepsilon$, $b=1$};\\
% 0, & \text{if $a = \varepsilon$, $b=-1$};\\
% \frac{1-1/p}{2p^3(1+1/p)(1-1/p^3)}, & \text{if $a = -\varepsilon$, $b=-1$};\\
% \frac{1-1/p}{4p(1-1/p^3)}, & \text{otherwise.}
% \end{cases}
% $$
% For $n=3$ we compute
% $$
% \PP_3\{a,b\}  = \begin{cases}
% \frac{1 - \frac{1}{2p^3}\brackets{1+1/p}}{2(1+1/p)(1-1/p^4)}, & \text{if $a = 1$ or $r$, $b=1$};\\
% \frac{1-1/p}{4p^3(1+1/p)(1-1/p^4)}, & \text{if $a = 1$ or $r$, $b=-1$};\\
% \frac{1 + 1/p -2/p^4}{4p(1+1/p)(1-1/p^4)}, & \text{if $a = p$ or $pr$, $b=\varepsilon$};\\
% \frac{1-1/p}{4p(1+1/p)(1-1/p^4)}, & \text{if $a = p$ or $pr$, $b=-\varepsilon$.}
% \end{cases}
% $$
% And for $n=4$
% $$
% \PP_4\{a,b\}  = \begin{cases}
% \frac{1 +  1/p^2 - 1/p^9 - 1/p^{11} -\frac12 1/p^3 + \frac12 1/p^8}{2(1+1/p)(1-1/p^{10})}, & \text{if $a = 1$ or $r$, $b=1$};\\
% \frac{1-1/p}{4p^3(1+1/p)(1-1/p^5)}, & \text{if $a = 1$ or $r$, $b=-1$};\\
% \frac{1-1/p^4}{4p(1+1/p)(1-1/p^5)}, & \text{if $a = p$ or $pr$};\\
% \frac{1-\frac{1}{2p^3}(1+1/p^2)}{2(1+1/p)(1+1/p^2)(1-1/p^5)}, & \text{if $a = r$, $b=1$;}\\
% \frac{1+1/p^2-2/p^5}{4(1+1/p)(1+1/p^2)(1-1/p^5)}, & \text{if $a = r$, $b=-1$.}
% \end{cases}
% $$

\section{Applications}
In this section we would like to illustrate that the distribution of equivalence classes is a versatile tool for approaching different kinds of questions about random matrices. Some of the results have already been established, and some are new, in both cases we give short straightforward proofs which do not require any additional knowledge other than the formulae we derive and some very general mathematical background.

We will write $\PP_n$ for the probability measure on $S_n(\ZZ_p)$. We will also use the standard Vinogradov notation with symbols $\sim$, $\asymp$ and $\ll$ as well as Big $O$ notation. 

\subsection{Determinant and rank}

The following result is not new (e.g. see Brent and McKay \cite{brent1988determinants}), but we would like to give a short proof using the elementary divisor distribution.

\begin{cor}\label{easydet}
Let $Q_n$ be a random $n \times n$ symmetric matrix with independent entries distributed according to the Haar measure on $\ZZ_p$, then
\[
\PP_n\{|\det Q_n|_p = p^{-k}\} = p^{-k}(1-p^{-n}) + O(p^{-3k/2})\,.
\]
\end{cor}
\begin{proof}
Let $\lambda$ be a partition with $0\le \lambda_1 \le \ldots \le \lambda_n$ and $m_t = \#\{i:\lambda_i=t\}$. Let $|\lambda| = \lambda_1 + \ldots + \lambda_n$ be the weight of $\lambda$ and $\ell(\lambda) = n-m_0 \le n$ be its length. Then
\[
\PP_n\{|\det Q_n|_p = p^{-k}\} = \sum_{|\lambda|=k} \prod_{i=1}^{n} p^{-\lambda_i(n-i+1)} \frac{\pi_{n}}{\beta_{m_0}\cdots \beta_{m_k}\cdots} =  \sum_{|\lambda| = k} f_n(\lambda)\,,
\]
where $f_n(\lambda)$ is as defined in (\ref{fnlambda}). Express the last part $\lambda_n = k - \lambda_1 - \ldots - \lambda_{n-1}$ so that
\[
\PP_n\{|\det Q_n|_p = p^{-k}\} = 
% p^{-k} \sum_{\lambda} \prod_{i=1}^{n-1} p^{-\lambda_i(n-i)} \frac{\pi_{n}}{\beta_{m_0}\cdots \beta_{m_k}\cdots} = 
p^{-k} (1-p^{-n}) \sum_{\lambda} f_{n-1}(\lambda) \gamma_\lambda\,,
\]
where the sum now ranges over $\lambda$ with the restriction $\lambda_{n-1} \le k - |\lambda|$. If $\lambda_{n-1} \neq k - |\lambda|$, we have $\gamma_{\lambda} = 1/\beta_1 = 1$, otherwise $\gamma_{\lambda} = \beta_{\lambda_{n-1}}/\beta_{\lambda_{n-1} + 1} \le 1/(1-p^{-2})$. Then
\[
\PP_n\{|\det Q_n|_p = p^{-k}\} = p^{-k} (1-p^{-n}) \brackets{\PP_{n-1}\{\mathcal{A}\} +O\brackets{ \PP_{n-1}\{\mathcal{B}\} }} \ll p^{-k}\,,
\]
where 
\begin{equation*}
  \mathcal{A} = \{\lambda_{n-1} < k - |\lambda|\};\quad  \mathcal{B} = \{\lambda_{n-1} = k - |\lambda|\}\,.
\end{equation*}
% Observe that $\PP_n\{|\det X_n| = p^{-k}\} \ll p^{-k}$, and thus $\PP_n\{|\det X_n| \le p^{-k}\} \ll p^{-k}$. 
For the complement of $\mathcal{A}$ we have $\PP_{n-1}\{\overline{\mathcal{A}}\} = \PP_{n-1}\{\lambda_{n-1} \ge k - |\lambda|\}$, hence
\[\PP_{n-1}\{\mathcal{A}\} +O(\PP_{n-1}\{\mathcal{B}\}) = 1 + O(\PP_{n-1}\{\lambda_{n-1} \ge k - |\lambda|\})\,.\]
The condition $\lambda_{n-1} \ge k - |\lambda|$ implies $|\lambda| \ge (k + \lambda_1 + \ldots + \lambda_{n-2})/2 \ge k/2$, thus
\[
\PP_{n-1}\{\lambda_{n-1} \ge k - |\lambda|\} \le \PP_{n-1}\{|\lambda| \ge k/2\}  = \PP_{n-1}\{|\det Q_{n-1}|_p \le p^{-k/2}\}\ll p^{-k/2}\,,
\]
and $\PP_n\{|\det Q_n|_p = p^{-k}\} = p^{-k} (1-p^{-n}) (1 + O(p^{-k/2}))$.
\end{proof}

% \begin{cor}\label{ffrank}
% Let $Q$ be a random $n \times n$ symmetric matrix with entries distributed independently and uniformly on $\FF_p$. Then the probability that its rank is $n-r$ equals
% \[
% \PP\{\rk\, Q = n-r\} = p^{-r(r+1)/2} \frac{\pi_n}{\pi_{r}\beta_{n-r}}\,.
% \]
% \end{cor}
% \begin{proof}
% The distribution of $\rk Q$ is simply the distribution of $m_0$ for a random symmetric matrix over $\ZZ_p$. 
% \end{proof}
% The result above means that knowing the size of the orthogonal group one can immediately obtain the distribution of rank, and the other way around. In fact, this is exactly how the sizes of the orthogonal groups over the finite fields were initially computed (see Carlitz \cite{carlitz1954representations} and MacWilliams \cite{macwilliams1969orthogonal}).

There are several definitions of rank of a matrix over a ring with zero divisors, here $R_m = \ZZ/p^m\ZZ$ when $m >1$. The standard rank of $Q$ over $R_m$ equals rank of its reduction $\overline{Q}$ mod $p$. The distribution of rank in this case is given by
\[
\PP_n\{\rk\, Q = n-r\} = p^{-r(r+1)/2} \frac{\pi_n}{\pi_{r}\beta_{n-r}}\,.
\]
One can also consider determinantal rank
\[
\text{rk}_d\, Q = \max\{j \le n: \text{there exists a non-zero minor of size $j$}\}\,.
\]
In the following corollary we derive the upper and lower bounds on the distribution of determinantal rank. We are not aware if this result has been written down explicitly, but our aim here is to provide a short straightforward proof.
\begin{cor}\label{pkrank}
Let $Q$ be a random $n \times n$ symmetric matrix with entries distributed independently and uniformly on $R_m$. Then as $m \rightarrow + \infty$,
\[
\PP_n\{\text{rk}_d\, Q \le n-1\}\ \sim\ \brackets{1-p^{-n}}p^{-m}\,,
\]
for $2 \le r\le \lceil\frac{n}{2}\rceil$,
\[
mr^{-2} p^{-m(2r-1)} \ll \PP_n\{\text{rk}_d\, Q \le n-r\}\ \ll\ m \log (\min\{m, n\})\ p^{-m(2r-1)}\,,
\]
and for $r > \lceil\frac{n}{2}\rceil$
\[
\PP_n\{\text{rk}_d\, Q \le n-r\}\ \ll\ p^{-mn}\,.
\]
All implied constants are independent of $n$, and are uniform for $p\ge 2$. 
\end{cor}
\begin{proof}
Let $Q_n \in S_n(\ZZ_p)$ be a random matrix defined as usual. Its reduction is distributed uniformly on $S_n(R_m)$. Determinantal rank is invariant under conjugation, so the probability that rank mod $p^m$ is at most $j-1$ equals $\PP_n\{k_1 + \ldots + k_j \ge m\}$. For $j =n$ the statement is follows from Corollary \ref{easydet}. For $j < n$ set $r := n-r+1\ge 2$, then for $\mathbf{k}_j = (k_1,\ldots,k_j)$ 
\begin{align*}
\PP_n\{|\mathbf{k}_j| = l\} &= \sum_{|\mathbf{k}_j| = l} \prod_{i\le n}\frac{1}{p^{k_i(n-i+1)}} \frac{\pi_n}{\beta_{m_0}\cdots \beta_{m_k}\cdots} \asymp \\
% \sum_{k_1 + \ldots + k_j = l} \prod_{i\le j}\frac{1}{p^{k_i(n-i+1)}} \prod_{i > j}\frac{1}{p^{k_i(n-i+1)}} \frac{\pi_n}{\beta_{m_0} \cdots \beta_{m_{k}}\cdots}= \\ 
&\asymp \sum_{|\mathbf{k}_j|= l} \prod_{i\le j}\frac{1}{p^{k_i(j-i+1)}} \frac{1}{p^{l(r-1)}} \frac{\pi_j}{\beta_{m_0} \cdots \beta_{m_{k_j}}} \prod_{i > j}\frac{1}{p^{k_i(r-(i-j))}} \frac{\pi_{r-1}}{\beta_{m_{k_{j+1}}}\cdots} \asymp \\
 &\asymp \frac{1}{p^{l(r-1)}} \sum_{\frac{l}{j} \le t\le l} \PP_j\{|\det Q_j|_p = p^{-l},\ k_j = t\} \PP_{r-1}\{k_1 \ge t\}\,.
\end{align*}
Here we used that $0 <\prod_{i=1}^{\infty} (1-2^{-i}) \le \pi_N(q) \le 1$ for all $N \in \NN$ and $q \ge 2$.
% $q$-binomial coefficients are always bounded for $0 \le J \le N$ with
% \[
% 1 \le \frac{\pi_N(q)}{\pi_J(q)\pi_{N-J}(q)} = \frac{\prod_{i \le J+1}^N (1 - q^{-{i}})}{\prod_{i=1}^{N-J} (1-q^{-i})} \le \frac{1}{\prod_{i=1}^{\infty} (1-q^{-i})} \le \frac{1}{\prod_{i=1}^{\infty} (1-2^{-i})}\,.
% \]
% where the infinite product is a positive number for all $q \in(0,1)$.
Conditioning on $s=m_t$,
\begin{align*}
\PP_j\{|\det Q_j|_p = p^{-l},\ k_j = t\}  
\ll \sum_{1 \le s \le \frac{l}{t}} \PP_{j-s}\{|\det Q_{j-s}|_p = p^{-(l-ts)}, k_{j-s} < t\} p^{-s(l-ts) - \frac{ts(s+1)}{2}}\,.
\end{align*}
The largest term is at $s = \lfloor \frac{l}{t}\rfloor$ and is $\asymp p^{-A}$ with $A = \frac{t}{2}\lfloor\frac{l}{t}\rfloor(\lfloor\frac{l}{t}\rfloor + 1) +(l - t\lfloor\frac{l}{t}\rfloor)(\lfloor\frac{l}{t}\rfloor + 1)$, so
\[
p^{-A}  \ll \PP_j\{|\det Q_j|_p = p^{-l},\ k_j = t\} \ll lt^{-1} p^{-A} \,.
\]
% Suppose for all $d < j$ there exists a positive constant $C$ such that
% $$
% \PP_d \{|\det X| = p^{-L},\ k_d \le T\} \le \frac{C}{p^{\frac{T}{2}\lfloor\frac{L}{T}\rfloor(\lfloor\frac{L}{T}\rfloor + 1) + (L - T\lfloor\frac{L}{T}\rfloor)(\lfloor\frac{L}{T}\rfloor + 1)}}.
% $$
% In particular, this is certainly true for $d=1$.
With a bit more work one can actually obtain a better upper bound matching the lower bound in order of magnitude. Together with $\PP_{r-1} \{k_1 \ge t\} = p^{-tr(r-1)/2}$ we have
\begin{align*}
\PP_n\{|\mathbf{k}_j| = l\} \gg \frac{1}{p^{l(n-j)}} \sum_{\frac{l}{j} \le t \le l} p^{-A-\frac{t}{2}r(r-1)}\,.
\end{align*}
The exponent of $p$ is a piecewise continuous function in $t$, and as a function in integer $t$ it attains its minimum $lr$ on all $t \in \big[\frac{l}{r};\frac{l}{r-1}\big]$. If $r \le \lceil \frac{n}{2}\rceil$, this interval is contained inside the range, and for fixed $r$
\[
\PP_n\{|\mathbf{k}_j| = l\} \gg lr^{-2} p^{-l(2r-1)}.
\]
Similarly the upper bound satisfies
\[
\PP_n\{|\mathbf{k}_j| = l\} \ll p^{-l(2r-1)} \sum_{\frac{l}{j} \le t \le l} l/t \ll p^{-l(2r-1)} l \log (\min\{l, j\})\,.
\]
If $r > \lceil \frac{n}{2}\rceil$, the minimum is attained at $t = \lceil \frac{l}{n-r+1}\rceil > \frac{l}{r-1}$, but the upper bound follows trivially. Summing over $l \ge m$ concludes the proof.
\end{proof}

% \begin{cor}
% Let $Q$ be a random $n \times n$ symmetric matrix with entries distributed independently and uniformly on $\ZZ/p^m\ZZ$. Then the probability that the last $r$ entries in its canonical form are zero is $\asymp\ p^{-\frac{mr(r+1)}{2}}$.
% \end{cor}
% \begin{proof}
% The probability in question is the probability that last $r$ elementary divisors of a random symmetric matrix over $\ZZ_p$ have exponents at least $m$.
% \[
% \PP_n\{k_i \ge m\quad \forall i \in [n-r+1,n]\}\ \asymp\ p^{-\frac{mr(r+1)}{2}}\,.
% \]
% \end{proof}

\subsection{Local-global problems}

Let $\mathcal{A}$ be a subset of $\ZZ^N$ such that it has density in the sense that
\[
\gamma (\mathcal{A}) := \lim_{H \rightarrow \infty} \frac{\#\{\mathbf{x} = (x_1,\ldots,x_N) \in \ZZ^N: 1\le x_i\le H,\ \mathbf{x} \in \mathcal{A}\}}{H^{N}}
\]
exists. We say that the probabilistic local-global principle holds if
$$
\gamma (\mathcal{A}) = \prod_{\text{$p$ prime}} \gamma_p(\mathcal{A})\,,
$$
where $\gamma_{\mathcal{A}}(p)$ is the local density of $\mathcal{A}$ over $\ZZ_p$, and the product may include $p=+\infty$ corresponding to the Euclidean contribution.

 There is a lot of ongoing work on determining when the local-global principle holds in density; for some general conditions see \cite{ekedahl1991infinite}, \cite{poonen1999cassels}. The set of integer $n \times n$ symmetric matrices is essentially $\ZZ^{n(n+1)/2}$, so we can use the local-global principle to compute the density of some interesting sets of integer symmetric matrices. To do that over all $\ZZ_p$, we will assume that Corollary \ref{symel} holds for $p=2$ (the proof would go similarly to the proof for odd primes, but we do not wish to deal with the canonical form of quadratic forms over $\ZZ_2$).

Let us first define the elementary divisors of an integer matrix. Let $X \in \M_n(\ZZ)$, then there exist $U,V \in \GL_n(\ZZ)$ such that $X = U \Sigma V$, where $\Sigma = \diag(a_1,\ldots,a_n)$ with $a_i \ge 0$ and $a_1 | a_2|\cdots |a_n$. The integers $a_i$ defined uniquely are the elementary divisors of $X$. If we look at $X$ as an element of $\M_n(\ZZ_p)$, then $\text{val}_p(a_i) = k_i$ as in the definition of the corresponding decomposition over $\ZZ_p$.

\begin{cor}\label{cycprob}
As $H \to \infty$ the proportion of integer $n\times n$ symmetric matrices with entries between $1$ and $H$ having its first $n-1$ elementary divisors equal to one tends to
\[
\prod_p \brackets{\frac{\pi_n}{\beta_n} + p^{-1} \frac{\pi_n}{\beta_{n-1}\pi_1}}\,.
\]
As $n \rightarrow \infty$ this expression tends to
\[
\prod_{i=1}^{\infty} \zeta(2i+1)^{-1} = (\zeta(3)\zeta(5)\cdots\zeta(2k+1)\cdots)^{-1} \approx 0.7935\,.
\]
\end{cor}
\begin{proof}
The set in question is 
\begin{equation*}
    V(\ZZ) := \{Q \in S_n(\ZZ): a_1=\ldots=a_{n-1} = 1\} = \{Q \in S_n(\ZZ): a_{n-1} = 1\}\,
\end{equation*}
since $a_j|a_{n-1}$ for all $j \le n-1$. Then its complement 
\begin{equation*}
    \overline{V}(\ZZ) = \{Q \in S_n(\ZZ): a_{n-1} > 1\}
\end{equation*}
 is a closed subscheme of $\mathbb{A}_{\ZZ}^{n(n-1)/2}$ of codimension $2$. We can use the local-global principle for this problem due to Ekedahl \cite{ekedahl1991infinite}. Moreover, this is the symmetric matrix version of the application Ekedahl gives in \cite{ekedahl1991infinite}. Over $\ZZ_p$ we have
\[
\gamma_p(V) = \PP\{k_1=\ldots=k_{n-1}=0\} =  \frac{\pi_n}{\beta_n} + p^{-1} \frac{\pi_n}{\beta_{n-1}\pi_1}= \begin{cases} \frac{\pi_n}{\beta_n} \brackets{1 - \frac{1}{p}}^{-1}\,, & \text{if $n$ is odd}\,;\\
\frac{\pi_n}{\beta_n} \brackets{(1 - \frac{1}{p})^{-1} - \frac{1}{p^{n}}}\,, & \text{if $n$ is even}\,.
\end{cases}
\]
As $n$ grows,
$
\prod_p \gamma_p(V) \to \prod_p \frac{\pi_{\infty}}{\beta_{\infty}\pi_1} =  \prod_{i=1}^{\infty} \zeta(2i+1)^{-1}
$.
% $$
% \bigg\vert \prod_p \frac{\pi_{\infty}}{\beta_{\infty}\pi_1} - \prod_p\brackets{ \frac{\pi_n}{\beta_n} + \frac{1}{p} \frac{\pi_n}{\beta_{n-1}\pi_1}}\bigg\vert \ll \sum_p \frac{1}{p^{n+1}} \ll \frac{1}{2^{n}},
% $$
% hence the product converges as $n\rightarrow \infty$.
\end{proof}
The condition $k_1=\cdots=k_{n-1}=0$ is equivalent to $\cok Q$ being cyclic. The asymptotic in Corollary \ref{cycprob} is known to be the upper bound on the probability to have a cyclic cokernel in case of a general distribution of the entries, see the work of Matchett Wood \cite{wood2017distribution}. Moreover, in the case of general matrices Nguyen and Matchett Wood \cite{nguyen2018random} proved that the limit as $n \to \infty$ is independent of the distribution of entries. It is reasonable to assume that this should also hold for symmetric matrices.

 We also present the following new corollary on the density of integer symmetric matrices with the property that there is no prime $p$ such that its square divides every minor of fixed size.

\begin{cor}
Let $r,n \in \NN$ with $1\le r\le n-1$. Then as $H \to \infty$ the proportion of integer $n\times n$ symmetric matrices with entries between $1$ and $H$ such that the greatest common divisor of its minors of size $n-r$ is square-free tends to
\[
\prod_p \brackets{\sum_{t=0}^{r+1} \frac{\pi_n}{\beta_{n-t}\pi_t} p^{-\frac{t(t+1)}{2}} - \frac{\pi_n}{\beta_{n-r-1}\pi_{r+1}} p^{-(r+1)(r+2)}} > 0\,.
\]
\end{cor}
\begin{proof}
A square-free number is not divisible by $p^2$ for any prime $p$. Let $V(\ZZ) \subseteq S_n(\ZZ)$ be the set of matrices with a square-free $(n-r)\times(n-r)$ minor. All minors are polynomial in the entries of the matrix, and the complement of $V$ is a closed subscheme of $\mathbb{A}_{\ZZ}^{n(n+1)/2}$ of codimension $(r+1)(r+2)/2\ge 3$. Let
\[
N_{\le H}(n,r) = \#\{Q \in S_n(\ZZ): 1\le q_{ij}\le H,\ \text{$Q$ has a square-free minor of size $n-r$}\}\,.
\]
Applying Ekedahl's sieve \cite{ekedahl1991infinite} gives
\[
\lim_{H\to \infty} \frac{N_{\le H}(n,r)}{H^{n(n+1)/2}} = \prod_p \gamma_p(V)\,,
\]
where $\gamma_n(p)$ is the measure of the set
\[
V_p = \{Q \in S_n(\ZZ_p):  k_1+\ldots+k_{n-r} \le 1\}\,.
\]
The complement of $V_p$ in $S_n(\ZZ_p)$ is 
\[
\overline{V}_p = 
% \{Q \in S_n(\ZZ_p): \text{all $(n-r)\times(n-r)$ minors are $0$ mod $p^2$}\} =
\{Q \in S_n(\ZZ_p): \text{rk}_d\, Q \le n-r-1 \text{ in $\ZZ/p^2\ZZ$}\}\,.
\]
There are two options over each $\ZZ_p$. Either $k_1 = \cdots = k_{n-r} = 0$ with probability
\[
a_p(V) = \sum_{t=0}^r \frac{\pi_n}{\beta_{n-t}\pi_t} p^{-\frac{t(t+1)}{2}}\,, 
\]
or $k_1=\cdots=k_{n-r-1} = 0$ and $k_{n-r}=1$ with probability
\[
b_p(V) = \frac{\pi_n}{\beta_{n-r-1}\pi_{r+1}} p^{-\frac{(r+1)(r+2)}{2}}(1 - p^{-\frac{(r+1)(r+2)}{2}})\,,
\]
so $\gamma_p(V) = a_p(V) + b_p(V)$. We have $\mu(\overline{V}_p) \ll p^{-(r+1)^2/2} \ll p^{-2}$, thus the product over primes converges to a positive number. Moreover, if $r$ is fixed, or grows as $r = r(n)$, then as $n \to \infty$
\[
\prod_p \gamma_p(V) \to \prod_p \frac{\pi_{\infty}}{\beta_{\infty}}  \brackets{\sum_{t=0}^{r+1} \frac{p^{-\frac{t(t+1)}{2}}}{\pi_t} -
\frac{1}{\pi_{r+1}} p^{-(r+1)(r+2)}}\,.
\]
In particular, if $r = 1$,
\[
\prod_p \gamma_p(V) \to \frac{\prod_p ( 1 + p^{-3}(1+1/p)^{-1} + p^{-4})}{\zeta(3)\zeta(5) \cdots \zeta(2k+1)\cdots} \approx 0.9581\,,
\]
and if $r = r(n) \to \infty$ as $n \to \infty$, then $\prod_p \gamma_p(V) \to 1$.
\end{proof}

\subsection{Isotropy problem}

In this subsection we compute the probability that a random quadratic form with independent coefficients distributed according to the Haar measure on $\ZZ_p$ is isotropic. A quadratic form $Q$ over a ring $R$ in $n$ variables is isotropic if there exists $\mathbf{t} \in R^n \setminus \{0\}$ such that $Q(t_1, \ldots, t_n) = 0$. Otherwise call $Q$ anisotropic.

It has been long known that a quadratic form over $\ZZ_p$ in at least $5$ variables always has a non-trivial zero, so we only consider $n \le 4$. Isotropy is invariant under conjugation by $\GL_n(\QQ_p)$, thus to determine if a quadratic form is isotropic it suffices to look at its invariants. In particular, if $p$ is an odd prime and $\varepsilon = \chi(-1)$, then
\begin{enumerate}[(a)]
    \item a binary quadratic form $Q$ over $\QQ_p$ is isotropic iff $d(Q) = \varepsilon$;
    \item a ternary quadratic form $Q$ over $\QQ_p$ is isotropic iff $c(Q) = \langle -1, d(Q)\rangle$;
    \item a quaternary quadratic form $Q$ over $\QQ_p$ is anisotropic iff $d(Q) =1$ and $c(Q) = -1$.
\end{enumerate}
For more details see Cassels' book \cite{cassels2008rational}.

The formulae for the probability of isotropy for all $n \le 4$ have been obtained by Bhargava, Cremona, Fisher, Jones, and Keating \cite{bhargava2016probability}. Here we aim to do two things: first, to give a very short proof using the distribution of $\QQ_p$-classes, and second, to reconstruct the recursive procedure provided in \cite{bhargava2016probability} from the distribution of $\ZZ_p$-classes and to illustrate that it can in fact be seen as the ``local'' GOE computation.

\begin{cor}
Let $q_n$ be the probability of isotropy of a random quadratic form in $n$ variables with coefficients distributed independently and uniformly on $\ZZ_p$. Set $t = 1/p$, then
\begin{equation*}
    \begin{aligned}
   & q_2 = \frac{1}{2}\,;\\
    & q_3 = 1 - \frac{t}{2(1+ t)^2}\,;\\
    & q_4 = 1 - \frac{t^3(1-t)}{4(1+t)^2(1-t^5)}\,.
    \end{aligned}
\end{equation*}
\end{cor}
\begin{proof}
By Lemma \ref{sigmalem} and \ref{deltalem}, and Theorem \ref{qfclass}
\begin{align*}
            q_2& = \sigma_2(\varepsilon) = \frac{(1 + \varepsilon^2t)(1 - \varepsilon^2 t^3)}{2(1 + t)(1-t^3)} = \frac{1}{2}\,;\\
             q_3& = \rho_3(1,1) + \rho_3(r,1) + \rho_3(p,\varepsilon) + \rho_3(pr,\varepsilon) = 1 - \frac{t}{2 (1 + t)^2}\,;\\
             q_4 &= 1 - \rho_4(1,-1) = 1 - \frac{t^3(1 - t)}{4(1+t)^2(1-t^5)}\,.
\end{align*}

One can avoid using the distribution of classes altogether by analysing which combinations of signs and parities of elementary divisors give the right class, and directly computing several geometric sums using Theorem \ref{symdensity}. For a quadratic form $Q$ from the equivalence class defined by $\Sigma$ and $S$ we have $d(Q) = \prod_i p^{k_i} S_{ii}$ and $c(Q) = \prod_{i < j} \left<p^{k_i} S_{ii}, p^{k_j} S_{jj}\right>$. Reduce the number of cases by only considering primitive quadratic forms, i.e. with $k_1 = 0$. For $n = 2$ we have $\text{val}_p(\det Q)$ even and $\prod s_i = \varepsilon$, so
\[
q_2 = \frac{1}{1 - t^3}\bigg(\frac{\pi_2}{\alpha_{2}^{\varepsilon}} + \sum_{\substack{j > 0 \\ j \text{ even}}} t^{j} \frac{\pi_2}{\alpha_{1}^{s} \alpha_{1}^{\varepsilon s}}\bigg) = \frac{1}{2}\,.
\]
For $n = 3$, $Q$ is anisotropic if either $\text{val}_p(\det Q)$ is even and $c(Q) = -1$, or $\text{val}_p(\det Q)$ is odd and $c(Q) = -\varepsilon$, so
\begin{align*}
1 - q_3 =\frac{1}{1 - t^6} \bigg( \frac{\pi_3}{\alpha_2^{-\varepsilon}\beta_1} \sum_{i \text{ odd}} t^{i}+
\frac{\pi_3}{\beta_1 \alpha_2^{-\varepsilon}} \sum_{i \text{ odd}} t^{3i}
+\frac{\pi_3}{\beta_1} \sum_{\substack{i < j \\ i,j \text{ odd}}}  t^{2i+j} \frac{1}{\alpha_1^{s} \alpha_1^{-s}} \,+ \\ + \frac{\pi_3}{\beta_1} \sum_{\substack{0< i < j \\ i+j \text{ odd}}}  t^{2i+j} \frac{1}{\alpha_1^{s} \alpha_1^{-s\varepsilon}}\bigg) = \frac{t}{2(1 + t)^2}\,.
\end{align*}
We omit the computation for $q_4$, which is more bulky with $7$ geometric sums over triple indices.
\end{proof}

Yet another option is to use the recurrence equation (\ref{recurr}) to reproduce the recursive computation from the work of Bhargava, Cremona, Fisher, Jones, and Keating \cite{bhargava2016probability}. Let $\mathcal{I}_n$ denote the set of ordered pairs $(a,b)$ such that a quadratic form $Q$ in $n$ variables with $d(Q) = a$ and $c(Q) = b$ is isotropic. Let $u \in\{1,r\}$ and $j \in \NN^0$, and let $T_j^u$ denote the operator acting on such pairs as
\[
T_j^u (a,b) = (a^\prime(j,u), b^\prime(j,u)) = (aup^j,b \varepsilon^{j(j-1)/2}\left<a,u\right>\left<p,au\right>^{j-1})\,.
\]
This corresponds to the change in the discriminant and the Hasse invariant of the rest of the quadratic form under rescaling by $p$. Note that $T_n^1T_n^1 = id$, and $T_n^1 \mathcal{I}_n = \mathcal{I}_n$.

Condition on the multiplicity $m_0$ and use the same observation as in \cite{bhargava2016probability} that $Q$ can only be anisotropic if its reduction $\overline{Q}$  mod $p$ doesn't have a non-trivial zero over $\FF_p$. This happens in two cases: if $m_0  = \rk\, \overline{Q} = 1$, or if $m_0 = 2$ and $s_0 = -\varepsilon$, i.e. its reduction mod $p$ is irreducible over $\FF_p$. If $m_0 \ge 3$ by the Chevalley-Warning theorem $\overline{Q}$ is always isotropic over $\FF_p$. A non-trivial zero over $\FF_p$ can be lifted to $\ZZ_p$ via Hensel's lemma. Set $t=1/p$, and write $\PP_n\{\mathcal{I}_n\}$ for the probability that a quadratic form in $n$ variables is isotropic, then
\begin{align*}
\PP_n\{\mathcal{I}_n\} =  
t^{\frac{(n-1)n}{2}} \frac{\pi_n}{\alpha_1^+ \pi_{n-1}} \PP_{n-1}\{T_{n-1}^{1} \mathcal{I}_n \} + 
t^{\frac{(n-1)n}{2}} \frac{\pi_n}{\alpha_1^- \pi_{n-1}} \PP_{n-1}\{T_{n-1}^{r} \mathcal{I}_n \}\, \\
t^{\frac{(n-2)(n-1)}{2}} \frac{\pi_n}{\alpha_2^{-\varepsilon} \pi_{n-2}} \PP_{n-2}\{T_{n-2}^{r\varepsilon} \mathcal{I}_n\} + t^{\frac{n(n+1)}{2}} \PP_n\{\mathcal{I}_n\} + \\
% \brackets{ 1 - p^{-\frac{n(n+1)}{2}} - p^{-\frac{(n-1)n}{2}} \frac{\pi_n}{\beta_1 \pi_{n-1}} - p^{-\frac{(n-2)(n-1)}{2}} \frac{\pi_n}{\alpha_2^{-\varepsilon} \pi_{n-2}}}.
\PP_n\{m_0\ge 3 \, \cup\, (m_0=2,s_0=-\varepsilon)\}\,.
\end{align*}
Comparing to the notation of Bhargava, Cremona, Fisher, Jones, Keating \cite{bhargava2016probability} we recognise
\begin{align*}
   & \xi_1^{(n)} = t^{\frac{(n-1)(n-2)}{2}} \frac{\pi_n}{\alpha_2^{-\varepsilon} \pi_{n-2}} = t^{\frac{(n-2)(n-1)}{2}} \frac{(1 - t^{n-1})(1 - t^n)}{2(1 + t)} = \PP_n\{m_0=2, s_0 = -\varepsilon\}\,;\\
   & \xi_2^{(n)} = t^{\frac{n(n-1)}{2}} \frac{\pi_n}{\pi_{n-1} \beta_1} =  t^{\frac{(n-1)n}{2}} \brackets{1 - t^{n}} = \PP_n\{m_0=1\}\,;\\
    & \xi_0^{(n)} = 1 - t^{\frac{n(n+1)}{2}} - \xi_1^{(n)} - \xi_2^{(n)} = 
    \PP_n\{m_0\ge 3 \, \cup\, (m_0=2,s_0=-\varepsilon)\}\,;\\
    % 1 - p^{-\frac{n(n+1)}{2}} - p^{-\frac{n(n-1)}{2}} \frac{\pi_n}{\beta_1 \pi_{n-1}} - p^{-\frac{(n-2)(n-1)}{2}} \frac{\pi_n}{\alpha_2^{-\varepsilon} \pi_{n-2}}\,;\\
   & \alpha_1^{(n)} = \PP_{n-2}\{T_{n-2}^{r\varepsilon} \mathcal{I}_n\} = \PP_n\{\text{$Q$ is isotropic given $m_0=2,\, s_0 = -\varepsilon$}\}\,;\\
  &  \alpha_2^{(n)} = \frac{1}{2} \brackets{ \PP_{n-1}\{T_{n-2}^{1} \mathcal{I}_n\} + \PP_{n-1}\{T_{n-2}^{r} \mathcal{I}_n\}} = \PP_n\{\text{$Q$ is isotropic given $m_0=1$}\}\,,
\end{align*}
where $\xi_i^{(n)}$ and $\alpha_i^{(n)}$ are as in \cite{bhargava2016probability}. Further conditioning in their recursive procedure similarly corresponds to conditioning on the multiplicity $m_1 = \#\{i: k_i = 1\}$ and signature $s_1$ of the initial quadratic form in $n$ variables in respective cases. For $\alpha_1^{(n)}$ we have
\begin{align*}
   \PP_{n-2}\{ T_{n-2}^{r\varepsilon} \mathcal{I}_n\} = \PP_{n-3}\{T_{n-3}^1T_{n-2}^{r\varepsilon} \mathcal{I}_n\} \frac{\pi_{n-2}}{\alpha_1^+\pi_{n-3}} t^{\frac{(n-3)(n-2)}{2}} + 
   \PP_{n-3}\{T_{n-3}^r T_{n-2}^{r\varepsilon} \mathcal{I}_n\} \frac{\pi_{n-2}}{\alpha_1^-\pi_{n-3}} t^{\frac{(n-3)(n-2)}{2}}+\\
   \PP_{n-4}\{T_{n-4}^{r\varepsilon} T_{n-2}^{r\varepsilon} \mathcal{I}_n\} \frac{\pi_{n-2}}{\alpha_2^{-\varepsilon}\pi_{n-4}}t^{\frac{(n-4)(n-3)}{2}} + t^{\frac{(n-1)(n-2)}{2}} \PP_{n-2}\{T_{n-2}^1 T_{n-2}^{r\varepsilon} \mathcal{I}_n\} =  \alpha_1^{(n)}\,.
\end{align*}
The last term has subsequent conditioning on $m_2$. By the Chevalley-Warning theorem $m_2 >0$ implies isotropy since solving $a_1x_1^2 + a_2x_2^2+ a_3(px_3)^2 = 0$ is the same as $a_1x_1^2 + a_2x_2^2+a_3x_3^2=0$. Thus
\begin{equation*}
    \PP_{n-2}\{T_{n-2}^1 T_{n-2}^{r\varepsilon}\mathcal{I}_n\} = (1-t^{\frac{(n-1)(n-2)}{2}}) + t^{\frac{(n-1)(n-2)}{2}} \PP_{n-2}\{T_{n-2}^1 T_{n-2}^1 T_{n-2}^{r\varepsilon}\mathcal{I}_n\}\,.
\end{equation*}
Here $\PP_{n-2}\{T_{n-2}^1 T_{n-2}^1 T_{n-2}^{r\varepsilon} \mathcal{I}_n\} = \PP_{n-2}\{T_{n-2}^{r\varepsilon} \mathcal{I}_n\} = \alpha_1^{(n)}$. The corresponding notation in \cite{bhargava2016probability} is
\begin{align*}
     \beta_1^{(n)} &= \PP_{n-4}\{T_{n-4}^{r\varepsilon} T_{n-2}^{r\varepsilon} \mathcal{I}_n\} = \PP_n\{\text{$Q$ is isotropic given}\ m_0=m_1=2,\, s_0= s_1 =-\varepsilon\}\,;\\
     \beta_2^{(n)} &= \frac{1}{2}\sum_{u\in\{1,r\}}\PP_{n-3}\{T_{n-3}^u T_{n-2}^{r\varepsilon} \mathcal{I}_n\} = \PP_n\{\text{$Q$ is isotropic given}\ m_0=2,\, s_0= -\varepsilon,\, m_1=1\}\,;\\
    \nu_1^{(n)}& = t^{\frac{(n-1)(n-2)}{2}} = \PP_n\{\text{$m_2=0 $ given}\ m_0=2,\, s_0= -\varepsilon,\, m_1=0\}\,;\\
     \nu_2^{(n)} &= 0\,;\\
    \nu_0^{(n)} &= 1-t^{\frac{(n-1)(n-2)}{2}}\, = \PP_n\{\text{$m_2\ge 1$ given}\ m_0=2,\, s_0= -\varepsilon,\, m_1=0\}\,.
\end{align*}
For $u \in\{1,r\}$ we have
\begin{align*}
        \PP_{n-1}\{T_{n-1}^{u} \mathcal{I}_n \} = t^{\frac{n(n-1)}{2}} \PP_{n-1}\{T_{n-1}^1 T_{n-1}^{u} \mathcal{I}_n\} + 
        \PP_{n-2}\{T_{n-2}^1 T_{n-1}^{u} \mathcal{I}_n \} \frac{\pi_{n-1}}{\alpha_1^+\pi_{n-2}}t^{\frac{(n-2)(n-1)}{2}} + \\
        \PP_{n-2}\{T_{n-2}^r T_{n-1}^{u} \mathcal{I}_n \} \frac{\pi_{n-1}}{\alpha_1^-\pi_{n-2}}t^{\frac{(n-2)(n-1)}{2}} + \PP_{n-3}\{T_{n-3}^{r\varepsilon}T_{n-1}^{u} \mathcal{I}_n\}\frac{\pi_{n-1}}{\alpha_2^{-\varepsilon}\pi_{n-3}} t^{\frac{(n-3)(n-2)}{2}}\,.
\end{align*}
The term $\PP_{n-1}\{T_{n-1}^1 T_{n-1}^{u} \mathcal{I}_n\}$ needs further conditioning on $m_2$. Similarly to the previous case, if $m_2 \ge 2$ or if $m_2=1$ and $s_0s_2=\varepsilon$, a quadratic form will be isotropic, hence
\begin{align*}
    \PP_{n-1}\{T_{n-1}^1 T_{n-1}^{u} \mathcal{I}_n\} = t^{\frac{(n-1)(n-2)}{2}} \frac{\pi_{n-1}}{\alpha_1^{-\varepsilon s_0}\pi_{n-2}} \PP_{n-2}\{T_{n-2}^{r\varepsilon u}T_{n-1}^1 T_{n-1}^{u} \mathcal{I}_n\} \,+ \\
    t^{\frac{n(n-1)}{2}} \PP_{n-1}\{T_{n-1}^{u} \mathcal{I}_n\} + 
    1 - t^{\frac{n(n-1)}{2}} - t^{\frac{(n-1)(n-2)}{2}} \frac{\pi_{n-1}}{\alpha_1^{-\varepsilon s_0}\pi_{n-2}}\,.
\end{align*}
Further, one can check that $T_{n-2}^{r\varepsilon u}T_{n-1}^1 T_{n-1}^{u} = T_{n-2}^{r\varepsilon}$ by the definition of $T_j^u$. The corresponding notation here is
\begin{align*}
        \gamma_1^{(n)} &= \frac{1}{2}\sum_{u\in\{1,r\}} \PP_{n-3}\{T_{n-3}^{r\varepsilon}T_{n-1}^{u} \mathcal{I}_n\}  = \PP_n\{\text{$Q$ is isotropic given}\ m_0=1,\, m_1=2,\, s_1 =-\varepsilon\}\,;\\
        % & \gamma_2^{(n)} =  \frac{1}{4}(\PP_{n-2}\{T_{n-2}^1 T_{n-1}^{1} \mathcal{I}_n \} + \PP_{n-2}\{T_{n-2}^1 T_{n-1}^{r} \mathcal{I}_n \} + \PP_{n-2}\{T_{n-2}^r T_{n-1}^{1} \mathcal{I}_n \} + \PP_{n-2}\{T_{n-2}^r T_{n-1}^{r} \mathcal{I}_n \})=\\
        \gamma_2^{(n)}& = \frac{1}{4} \sum_{u,v \in\{1,r\}} \PP_{n-2}\{T_{n-2}^u T_{n-1}^{v} \mathcal{I}_n \} = \PP_n\{\text{$Q$ is isotropic given}\ m_0=m_1=1\}\,;\\
        \eta_1^{(n)} &= \frac{1}{2} t^{\frac{(n-1)(n-2)}{2}} (1 - t^{n-1}) = \PP_n\{\text{$m_2=1,s_2=-\varepsilon s_0 $ given}\ m_0=1,\, m_1 = 0\}\,;\\
        \eta_2^{(n)} &= t^{\frac{n(n-1)}{2}} = \PP_n\{\text{$m_2=0 $ given}\ m_0=1,\, m_1 = 0\}\,;\\
        \eta_0^{(n)} &= 1 -\eta_1^{(n)} -\eta_2^{(n)}\,.
\end{align*}
Finally, the coefficients $\beta_i^{(n)}$ and $\gamma_i^{(n)}$ are easy to compute for small $n$. In particular, if $n \le 4$, the respective coefficients come from quadratic forms in $\le 2$ variables.

This way of interpreting the recursive procedure from \cite{bhargava2016probability} also brings together two seemingly different approaches for computing probability of isotropy over $\ZZ_p$ and $\RR$ in case of the Gaussian Orthogonal Ensemble. First, both over $\ZZ_p$ and $\RR$ the computation is done for ``Gaussian`` invariant ensembles. Second, all local computations rely on the joint distribution of the canonical form of a quadratic form over the corresponding local field with the difference that eigenvalues over $\RR$ are almost surely distinct, and over $\ZZ_p$ elementary divisors are allowed to repeat.

\section*{Acknowledgements}
The author was funded by ERC LogCorRM (grant no 740900) at University of Oxford. Part of the work was carried out at University of Bristol. The author is very grateful to Jon Keating for his comments on the draft of this paper as well as his support and guidance, and to Zeev Rudnick for a helpful discussion. The author would also like to thank Melanie Matchett Wood for comments on the draft and for pointing to additional references, and the referee for their suggestions and corrections. Finally, the author is very grateful to Evan O'Dorney for reaching out and providing a thorough review of the paper and pointing out several errors and oversights.

\bibliographystyle{alpha}
\bibliography{biblio}{}

\end{document}